\documentclass[12pt]{amsart}  
\usepackage{amssymb,amsmath,amsthm,amscd}
\usepackage[all]{xy}

\usepackage[hypertex]{hyperref}

\numberwithin{equation}{section}

\newtheoremstyle{fancy1}{10pt}{10pt}{\itshape}{12pt}{\textsc\bgroup}{.\egroup}{8pt}{
}
\newtheoremstyle{fancy2}{10pt}{10pt}{}{12pt}{\itshape}{.}{8pt}{ }

\theoremstyle{fancy1}

\newtheorem{cor}[equation]{Corollary}

\newtheorem{prop}[equation]{Proposition}
\newtheorem{thm}[equation]{Theorem}

\newtheorem*{main*}{Theorem}

\newtheorem*{cor*}{Corollary}

\newtheorem*{problem*}{Problem}

\renewcommand{\thetable}{\theequation}
\theoremstyle{fancy2}

\newtheorem*{rem*}{Remark}

\newtheorem*{rems*}{Remarks}

\newcommand{\cref}[1]{Corollary~\ref{#1}}

\newcommand{\pref}[1]{Proposition~\ref{#1}}

\newcommand{\tref}[1]{Theorem~\ref{#1}}
\newcommand{\sref}[1]{Section~\ref{#1}}

\newcommand{\gs}{\sigma}

\newcommand{\e}{\epsilon}

\newcommand{\RP}{\mathbb{R\mkern1mu P}}
\newcommand{\CP}{\mathbb{C\mkern1mu P}}

\newcommand{\Sph}{\mathbb{S}}

\newcommand{\C}{{\mathbb{C}}}
\newcommand{\R}{{\mathbb{R}}}
\newcommand{\Z}{{\mathbb{Z}}}
 \newcommand{\Q}{{\mathbb{Q}}}

\newcommand{\QH}{{\mathbb{H}}}

\renewcommand{\H}{\ensuremath{\operatorname{H}}}

\newcommand{\SO}{\ensuremath{\operatorname{SO}}}

\newcommand{\Sp}{\ensuremath{\operatorname{Sp}}}
\newcommand{\U}{\ensuremath{\operatorname{U}}}
\newcommand{\SU}{\ensuremath{\operatorname{SU}}}

\newcommand{\T}{\ensuremath{\operatorname{T}}}
\renewcommand{\S}{\ensuremath{\operatorname{S}}}

\def\con#1=#2(#3){#1 \equiv #2 \bmod{#3}}

\newcommand{\sign}{\ensuremath{\operatorname{sign}}}
\newcommand{\signs}{\ensuremath{\operatorname{sgn}}}

\newcommand{\tr}{\ensuremath{\operatorname{tr}}}
\newcommand{\diag}{\ensuremath{\operatorname{diag}}}

\DeclareMathOperator{\Id}{Id} 
 
\newcommand{\no}{\noindent}

\begin{document}

\title{Topology of non-negatively curved manifolds}

\author{Christine Escher}
\address{Oregon State University\\
   Corvallis, OR 97331}
\email{tine@math.orst.edu}
\author{Wolfgang Ziller}
\address{University of Pennsylvania\\
   Philadelphia, PA 19104}
\email{wziller@math.upenn.edu}
\thanks{
The first author was supported by a grant from the Association for
Women in Mathematics, by the University of Pennsylvania and by IMPA.
The second
 author was supported by  a grant from the National Science Foundation, the Max Planck Institute
 in Bonn and CAPES. }

\maketitle

An important question in the study of Riemannian manifolds of
positive sectional curvature is how to distinguish manifolds that
admit a metric with non-negative sectional curvature from those that
admit one of positive curvature. Surprisingly, if the manifolds are
compact and simply connected, all known obstructions to positive
curvature are already obstructions to non-negative curvature. On the
other hand, there are  very few known examples of manifolds with
positive curvature. They consist, apart from the rank one symmetric
spaces, of certain homogeneous spaces $G/H$ in dimensions
$6,7,12,13$ and $24$ due to Berger \cite{Be}, Wallach \cite{Wa}, and
Aloff-Wallach \cite{AW}, and of biquotients $K\backslash G/H$ in
dimensions $6,7$ and $13$ due to Eschenburg \cite{E1},\cite{E2} and
Bazaikin \cite{Ba}, see \cite{Zi} for a survey. Recently, a new
example of a positively curved 7-manifold was found which is
homeomorphic but not diffeomorphic to the unit tangent bundle of $\Sph^4$, see
\cite{GVZ,De}. And in \cite{PW} a method was proposed  to construct
a metric of positive curvature on the Gromoll-Meyer exotic 7-sphere.

Among the known examples of positive curvature there are two
infinite families: in dimension 7 one has the homogeneous
Aloff-Wallach spaces,  and more generally the Eschenburg
biquotients, and in dimension 13 the Bazaikin spaces. The topology
of these manifolds has been studied extensively, see
\cite{KS1,KS2,AMP1,AMP2,Kr1,Kr2,Kr3,Sh,CEZ,FZ}. There exist many
7-dimensional positively curved examples which are homeomorphic to
each other but not diffeomorphic, whereas in dimension 13, they are
conjectured to be diffeomorphically distinct \cite{FZ}.

In contrast to the positive curvature setting, there exist
comparatively many examples with non-negative sectional curvature.
The bi-invariant metric on a compact Lie group $G$ induces, by
O'Neill's formula, non-negative curvature on any homogeneous space
$G/H$ or more generally on any biquotient $K\backslash G/H$. In
\cite{GZ1} a large new family of cohomogeneity one manifolds with
non-negative curvature was constructed, giving rise to
non-negatively curved metrics on exotic spheres. Hence it is natural
to ask whether, among the known examples,  it is possible to
topologically distinguish manifolds with non-negative curvature from
those admitting positive curvature. The purpose of this article is
to address this question. There are  many examples of non-negatively
curved manifolds which are not homotopy equivalent to any of the
known positively curved examples simply because they have different
cohomology rings. But recently new families  of non-negatively
curved manifolds were discovered \cite{GZ2} which, as we will see,
give rise to several new manifolds having the same cohomology ring
as the 7-dimensional Eschenburg spaces.
\smallskip

Recall that
the Eschenburg biquotients are defined as
$$
E_{k,\; l}=\diag(z^{k_1}, z^{k_2}, z^{k_3})\backslash \SU(3)/
\diag(z^{l_1}, z^{l_2}, z^{l_3})^{-1},
$$
with $k:=(k_1,k_2,k_3),\,l:=(l_1,l_2,l_3),\,k_i,l_i \in \Z$,
{\small$\sum$} $\!k_i=$ {\small$\sum$} $\!l_i$,  and $z \in S^1
\subset \C$.  They include the homogeneous Aloff-Wallach spaces
$W_{a,b}=\SU(3)/\diag(z^{a},z^{b},\bar{z}^{a+b})$ and the manifolds
$F_{a,b}$ with $k=(a,b,a+b)$ and $l=(0,0,2(a+b))$. Under certain
conditions on $k,l$, see \eqref{pos}, they admit positive sectional
curvature.

We will consider compact simply connected seven dimensional manifolds $M$  whose non-trivial cohomology  groups consist of
 $H^i(M;\Z)\cong \Z$ for $i=0,2,5,7$ and
 $H^4(M,\Z)\cong \Z_r,\ r\ge 1$, where the square of a generator of $H^2(M; \Z)$
 generates $H^4(M,\Z)$. If in addition $M$ is spin, we say that $M$ has {\it cohomology type $E_r$}
and if $M$ is non-spin, {\it cohomology type $\bar{E}_r$}.
  In \cite{E1} it was shown that an
 Eschenburg space
 is of cohomology type $E_r$ with  $r= |\sigma_2(k)-\sigma_2(l)|$, where $\sigma_i(k)$
stands for the elementary symmetric polynomial of degree $i$ in
$k_1,k_2,k_3$.   M. Kreck and S. Stolz defined certain invariants and showed that they classify manifolds of
cohomology type $E_r$ and $\bar{E}_r$ up to homeomorphism and diffeomorphism \cite{Kr,KS1,KS2}.
The invariants were computed for most Eschenburg spaces in \cite{Kr3,CEZ}.

\smallskip

One class of manifolds we will study are  the total spaces of 3-sphere bundles over $\CP^2$. They fall into two categories, bundles which are not spin,
$$ \Sph^3\to S_{a,\, b }\to \CP^2 ,\, p_1=2a+2b+1 ,\, e=a-b  ,\, w_2\ne 0 $$
and bundles which are spin
$$ \Sph^3\to \bar S_{a,\, b }\to \CP^2 ,\, p_1=2a+2b  ,\, e=a-b  ,\, w_2 = 0 .$$
They are specified by the value of their Pontryagin class, Euler class and Stiefel-Whitney class. Here $a,b$ are arbitrary integers.

\smallskip

A second class of manifolds can be described
as follows. Consider the bundles
$$
\CP^1 \to N_t\to \CP^2 \quad , \quad \Sph^1\to M^t_{a,\, b }\to  N_t \,.
$$
Here $N_t$ is the $\Sph^2$ bundle with Pontryagin class $p_1=1-4t$
and $w_2\ne 0$, and $M^t_{a,\, b }$  the circle bundle classified by
the Euler class  $e=ax+by$ in terms of some natural basis $x,y\in
H^2(N_t,\Z)\cong\Z^2$, where $a,b$ are relatively prime integers.

Similarly, let
$$
\CP^1 \to \bar N_t\to \CP^2 \quad , \quad \Sph^1\to \bar M^t_{a,\, b }\to  N_t \,
$$
where $\bar N_t$ is the $\Sph^2$ bundle with Pontryagin class $p_1=4t,\ w_2=0 $, and $\bar M^t_{a,\, b }$ the circle bundle with Euler class described by  $a,b$ as above.
\bigskip
We will show:

\begin{main*}The above manifolds have the following properties:
  \begin{itemize}
\item[(a)]  $S_{a,\, b },\ M^t_{a,\, b }$ and $\bar M^t_{a,\, 2b }$ have cohomology type $E_r$, and $\bar S_{a,\, b },\ \bar M^t_{a,\, 2b+1 }$
have cohomology type $\bar E_r$.
 \item[(b)] $S_{a,\, b },\ M^t_{a,\, b }$ and $\bar M^{2t}_{a,\, b }$ admit metrics with non-negative sectional curvature.
\item[(c)]  $M^t_{a,1-a}$ is  diffeomorphic to $S_{-t,a(a-1)}$, and $\bar M^t_{a,1}$ is  diffeomorphic to $\bar S_{t, a^2}$.
 \item[(d)]  $M^1_{a,\, b }$ is the Aloff-Wallach space $W_{a,\, b }$ with base space $N_1=\SU(3)/\T^2$,
  and $M^{-1}_{a,\, b }=F_{a,b}$ with base space the biquotient
  $N_{-1}=\SU(3)/\!/ \T^2$.
 \item[(e)] $ M^0_{a,\, b }$ is the set of circle bundles over
 $N_0=\CP^3\#\overline{\CP^3}$ and  $\bar M^0_{a,\, b }$  the
 set of circle bundles over $\bar N_0=\CP^2\times\CP^1$.
 \end{itemize}
\end{main*}

\smallskip
The existence of the metrics in part (b) will follow from \cite{GZ2}
after we describe the manifolds in a different fashion, namely as
quotients of certain $\U(2)$ principal bundles over $\CP^2$. Part
(c) and (d) imply that the circle bundles $M^1_{a,b},\ ab(a+b)\ne
0$, and $M^{-1}_{a,b},\ ab>0$  as well as the sphere bundles $
S_{-1,a(a-1)},\ a\ge 2$ naturally admit a metric with positive
sectional curvature. By computing the Kreck-Stolz invariants, we
obtain a diffeomorphism classification of the above four classes of
7-manifolds and by comparing them to the invariants for Eschenburg
spaces, we will obtain many other diffeomorphisms of $M^t_{a,b}$ and
$S_{a,b}$ to positively curved Eschenburg spaces. For example:

\begin{itemize}
\item
 $S_{a,\, b }$ with $a-b=41$ is diffeomorphic to a positively
curved  Eschenburg space   if and only if $b\equiv 2285 \text{
or } 5237 \mod 6888 $. In this case it is diffeomorphic to the
cohomogeneity two Eschenburg space $E_{k,l}$ with $k=(2,3,7)\; ,
\; l=(12,0,0)$.
\item
$M^t_{a,\, b }$ with $(a,b,t) = ( 638,-607, -403)$ is
diffeomorphic to the positively curved cohomogeneity two
Eschenburg space $E_{k,l}$ with $k=(1, 2, 5)\; , \; l=(8,0, 0)$.
 \end{itemize}

See \sref{examples} and Table A and B for further examples. We will
also obtain a description of which Eschenburg spaces can be
diffeomorphic to $\Sph^3$ bundles over $\CP^2$, see \tref{Ediffeo}.

\smallskip

The  examples of diffeomorphisms above imply that besides the metric
of non-negative curvature the  manifolds $S_{a,b}$ and $M^t_{a,b}$
sometimes admit a very different metric which has positive
curvature.  This raises the question whether perhaps they all  admit
a metric of positive curvature.

\smallskip

Some of these manifolds are also known to admit Einstein metrics. In
\cite{W} M.Wang showed that the Aloff-Wallach spaces $M^1_{a,\, b
}$, and with W.Ziller in \cite{WZ} that the circle bundles $\bar
M^0_{a,\, b }$ over $\CP^1\times\CP^2$, all admit Einstein metrics.
In \cite{Che} D. Chen proved that the sphere bundles $S_{a,\, b }$
and $\bar S_{a,\, b }$ admit an Einstein metric if the structure
group reduces from $\SO(4)$ to $\T^2\subset\SO(4)$. Diffeomorphisms
within each of these 3 classes of Einstein manifolds have been
considered in \cite{KS1,KS2,Che}. Using our computation of the
invariants, we also find examples of diffeomorphism between
different classes:

\begin{itemize}
\item The two Einstein manifolds  $\bar M^0_{ 70, 5899 }$ and $\bar S_{ 62500, 57600
}$ are diffeomorphic to each other.
\item For the sphere bundle $\bar S_{a,\, b }$ with $(a,b)=(q^2,0)$
the structure group reduces to a 2-torus. This Einstein manifold
is  diffeomorphic to the Einstein manifold $\bar M^0_{q,\, 1 }$.
 \end{itemize}

 We will also see that among some of these classes there are no
 diffeomorphisms:

\smallskip

\begin{itemize}
\item There are no diffeomorphisms between the spin Einstein manifolds $M^1_{a,\, b}$ and
$\bar M^0_{a',\, 2b' }$.
\item There are no diffeomorphisms between the spin Einstein manifolds
$S_{a\, ,b}$  and either $M^1_{a',\, b'}$ or $\bar M^0_{a',\,
2b' }$ or an Eschenburg space.
 \end{itemize}

\smallskip

Here is a short description of the content of the paper. In Section 1 we collect preliminaries and in Section 2 we recall the
formulas for the Kreck-Stolz invariants. In Section 3 we
describe  the topology of the sphere bundles $S_{a,\, b }$ and their invariants and   in Section 4   the topology of the
circle bundles $M^t_{a,\, b }$.  In Section 5 we discuss the geometry of both families and
in Section 6 we examine the manifolds  $\bar{S}_{a,b}$ and $\bar{M}^t_{a,b}$.   In Section 7 we apply these results to
obtain various examples of diffeomorphisms as described above.

\section{Preliminaries}

We will compare several classes of manifolds of non-negative
curvature with the family of positively curved Eschenburg spaces
$E_{k,l}$ described in the introduction.  In order for an
Eschenburg space to be a manifold, i.e. in order for the $\S^1$
action on $\SU(3)$ to be free, we need
\begin{equation}\label{free}
\gcd(k_1-l_{i}\ , k_{2}-l_{j})=1, \text{ for all } i\neq j\,,\,i,j \in \{1,2,3\}\,.
\end{equation}

The {\it Eschenburg metric} on $E_{k,l}$ is the submersion metric
obtained by scaling the bi-invariant metric on $\SU(3)$ in the
direction of a subgroup  $\U(2)\subset\SU(3)$ by a constant less
than $1$. It has positive sectional curvature \cite{E2} if and only
if, for all $1\leq i \leq 3$,
\begin{equation}\label{pos}
k_i \notin [\min(l_1,l_2,l_3),\max(l_1,l_2,l_3)], \ \ \text{or}\ \
l_i \notin [\min(k_1,k_2,k_3),\max(k_1,k_2,k_3)].
\end{equation}
If this condition is satisfied, we call $E_{k,l}$ a {\it
positively curved Eschenburg space}.

\smallskip

There are two subfamilies of Eschenburg spaces that are of interest to us.
One is the family of
 homogeneous Aloff-Wallach spaces
$W_{p,\; q}=\SU(3)/\diag(z^{p},z^{q},\bar{z}^{p+q})$, where $p, q
\in \Z$ with $(p,q)=1$.  This space has a homogeneous metric with
positive sectional curvature if and only if $pq(p+q)\ne 0$. By
interchanging coordinates, and replacing $z$ by  $\bar{z}$ if
necessary, we can assume that $p\ge q \ge 0$, and thus $W_{1,0}$ is
the only Aloff-Wallach space that does not admit a homogeneous
metric with positive curvature. The second family consists of the
Eschenburg biquotients $F_{p,\; q}=E_{k,l}$ with $k=(p,q,p+q)$ and
$l=(0,0,2p+2q)$ with $(p,q)=1$.  We can also assume that $p\ge q$,
but here $p$ and $q$ can have opposite sign. The Eschenburg metric
on $F_{p,\; q}$ has positive sectional curvature if and only if $p\,
q> 0$. These two families of Eschenburg spaces are special in that
they admit circle fibrations over positively curved 6-manifolds. In
the case of the Aloff-Wallach spaces, the base of the fibration is
the homogeneous flag manifold $\SU(3)/\T^2$ and in case of the
Eschenburg biquotients $F_{p,\; q}$, the base is the inhomogeneous
Eschenburg flag manifold
$\SU(3)/\!/\T^2:=\diag(z,w,zw)\backslash\SU(3)/\diag(1,1,z^2w^2)^{-1}$,
$|z|=|w|=1$.

We  also use the fact that   the two Eschenburg spaces $W_{1,1}$ and
$F_{1,1}$ can be regarded as principal $\SO(3)$ bundles over
$\CP^2$, see \cite{Sh,Cha}.

\bigskip

Recall that we say that a compact simply connected seven dimensional spin manifold $M$ has {\it cohomology type} $E_r$, or simply is
of type $E_r$, if its cohomology ring is given by:
\begin{equation}\label{coh}
  H^0(M;\Z)\cong H^2(M;\Z)\cong H^5(M;\Z)\cong
  H^7(M;\Z)\cong \Z
  \text{ and } H^4(M;\Z)\cong \Z_r.
\end{equation}
with $r \ge 1$. Furthermore, if $u$ is a generator of
$H^2(M;\Z)$, then $u^2$ is a generator of $H^4(M;\Z)$, if $r>1$. If a manifold with cohomology ring \eqref{coh} is non-spin, we say it has type $\bar{E}_r$.

\smallskip

As was shown in  \cite{E1}, an Eschenburg space is of type $E_r$
with $r= |\sigma_2(k)-\sigma_2(l)|$ where $\sigma_i(k)$ stands for
the elementary symmetric polynomial of degree $i$ in $k_1,k_2,k_3$.
Furthermore, for an Eschenburg space $r$ is always odd (see
\tref{Ktop}), and examples of Eschenburg spaces exist for any odd
number $r\ge 3$, for example the cohomogeneity one manifolds
$k=(p,1,1)$ and $l=(p+2,0,0)$ with $r=2p+1$.
\smallskip

Let $M$ be a  manifold of cohomology type $E_r$ or $\bar{E}_r$.  If we fix
  a generator $u$ of $H^2(M;\Z)$, the class
$u^2$ is  a generator of $H^4(M;\Z)$, which does not depend on
the choice of $u$. We can thus identify the first Pontryagin class
$p_1(TM)\in H^4(M;\Z)\cong \Z_r$ with a well defined integer modulo
$r$. Furthermore, it is a
homeomorphism invariant. In the case of an Eschenburg space this integer is
$p_1(TE_{k,l})  = 2\,\sigma_1(k)^2 - 6\,\sigma_2(k) \mod r$, see \cite{Kr2}.

\smallskip

For manifolds of type $E_r$ or $\bar{E}_r$ we also have, besides $r$, a second
homotopy invariant given by the linking form $Lk$.  The linking form
is a quadratic form on $ H^4(M;\Z)\cong \Z_r$ with values in $\Q /
\Z$.  It is determined by the self-linking number
  $lk(M):=Lk(u^2,u^2)$ of the generator $u^2$ of $ H^4(M;\Z)$.  Following \cite{B} $Lk(u^2,u^2)$ can
  be written in terms of the Bockstein homomorphism $\beta: H^3(M;\Q/\Z) \longrightarrow H^4(M;\Z)$ which
  is associated to the exact sequence $0 \longrightarrow \Z \longrightarrow \Q \longrightarrow \Q/\Z \longrightarrow 0:$
$$Lk(u^2,u^2):=<\beta^{-1}(u^2),u^2 \cap [M]>.$$
Since in our case $H^4(M;\Z) \cong \Z_r$, it follows that $Lk(u^2,u^2)$ is
a unit in the subgroup of order $r$ in $\Q/\Z $ and we can thus
interpret $lk(M)$ as an integer modulo $r$. Unlike $p_1(TM)$, the
linking form $lk(M)$ is orientation sensitive. In the case of an
Eschenburg space we have $lk(E_{k,l})=\pm s^{-1} $ mod $r$, where
$s=\sigma_3(k)-\sigma_3(l)$ and $s^{-1}$ is the multiplicative
inverse in $\Z_r$, which is well defined since the freeness
condition implies $(s,r)=1$, see \cite{Kr2}.

\smallskip

As was observed by Kruggel, manifolds of type $E_r$ or $\bar E_r$
fall into two further homotopy types since $\pi_4(M)$ can only be
$0$ or $\Z_2$. Indeed, if we consider the circle bundle $\S^1\to
G\to M$ whose Euler class is a generator of $H^2(M,\Z)=\Z$, it
follows that $H^*(G,\Z) \cong H^*(\Sph^3\times\Sph^5,\Z)$ and hence the
attaching map of the 5 cell to the 3-skeleton is an element of
$\pi_4(\Sph^3) \cong \Z_2$. If the attaching map is trivial, $G$ has the
homotopy type of $\Sph^3\times\Sph^5$ and hence $\pi_4(M) \cong \Z_2$. A
second manifold $G$ with this homology is $G=\SU(3)$  and since
$\pi_4(\SU(3))=0$ , its attaching map is non-trivial. Thus
$\Sph^3\times\Sph^5$ and $\SU(3)$ are the only two possible homotopy
types for $G$.

 For
Eschenburg spaces we have $\pi_4(E_{k,l})=0$ since
$\pi_4(\SU(3))=0$.

\bigskip

Lastly, we discuss the relationship between principal $\SO(3)$ and
$\SO(4)$  bundles and their classification if the base is $\CP^2$,
see  \cite{DW} and \cite{GZ2}. Recall that
$\SO(4)=\S^3\times\S^3/\{\pm (1,1)\}$ defined by left and right
multiplication of unit quaternions on $\QH\cong\R^4$. Thus there are
two normal subgroups $$\S^3_-=\S^3\times\{e\}\, ,\,
\S^3_+=\{e\}\times\S^3\subset\SO(4)=\S^3\times\S^3/\{\pm (1,1)\}$$
isomorphic to $\S^3$ and $\SO(4)/\S^3_\pm$ is isomorphic to
$\SO(3)$. Hence, if $\SO(4)\to P \to M$ is a principal $\SO(4)$
bundle, there are two associated principal $\SO(3)$-bundles $$\SO(3)
\to P_\pm :=P/\S^3_\pm  \to M \text{ with } \SO(3) =  S^3/\{\pm
1\}.$$
 If $M$ is compact and simply connected,
 $P$ is uniquely determined by the $\SO(3)$ bundles $P_\pm$, see
\cite{GZ2}, Proposition 1.8. For the characteristic classes one has
\begin{equation}\label{ppm}
p_1(P_\pm)=p_1(P)\pm 2\,e(P) \; , \;  w_2(P)=w_2(P_\pm)\,.
\end{equation}
\noindent One can see this on the level of classifying spaces  by
computing the maps induced  in cohomology in the commutative diagram
\begin{equation*}
\begin{split}
\xymatrix{
B_{\S^3\times\S^3 }\ar[r] \ar[d]^{\pi_\pm} & B_{\SO(4)} \ar[d]^{/\S^3_\pm}\\
B_{\S^3 } \ar[r] & B_{\SO(3)}  }
\end{split}
\end{equation*}
where $\pi_\pm$ are induced by the projections onto the first and
second factor. On the level of maximal tori one has
$\diag(e^{i\theta},e^{i\psi})\subset \S^3\times\S^3 \to
\diag(R(\theta-\psi),R(\theta+\psi))\subset \SO(4)$ where
$R(\theta)$ is a rotation by angle $\theta$. Thus in the natural
basis of the second cohomology of the maximal tori, $x,y$ in the
case of
 $\SO(4)$ and $r,s$ in the case of
$\S^3\times\S^3$, one has $x\to r-s$ and $y\to r+s$ and since
$p_1(P)=x^2+y^2$ and $e(P)=xy$, we have
 $p_1(P) +2 \, e(P)=4\, r^2$ and
$p_1(P) -2 \, e(P)=4\, s^2$. The claim now follows by observing that the map
$H^4(B_{\SO(3)},\Z)\cong \Z \to H^4(B_{\S^3},\Z)\cong \Z$ is multiplication by $4$.
 Notice that this corrects a mistake in
\cite{GZ2}, (1.10), where $P_\pm$ was defined as $P/\S^3_\mp$.
 This will be crucial for us in Section 6.

\smallskip

It is also important for us to understand in detail the above
discussion in the context of $\U(2)$ principal bundles. Assume that
the structure group of an $\SO(4)$ principal bundle $P$ reduces to
$\U(2)$:
 $$\U(2)\to P^*\to M
\text{ and } P=P^*\times_{\U(2)}\SO(4),$$
 or equivalently the 4-dimensional
vector bundle corresponding to $P$ has a complex structure. We
identify $\C\oplus\C\cong \QH$ via $(u,v)\to u+vj$ so that left
multiplication by $z$ is the usual complex structure on $\R^4$. This
defines the embedding $\U(2)\subset\SO(4)$ and implies that
$\U(2)=\S^1\times\S^3/\{\pm (1,1)\} \subset S^3\times\S^3/\{\pm
(1,1)\} = \SO(4)$. Notice also that the image of $\S^1\times\{e\}$
is the center of $\U(2)$, and the image of $\{e\}\times\S^3$ is
$\SU(2)\subset \U(2)$.

\smallskip

For $P^*$ we have the Chern classes $c_1$ and $c_2$ and for the
underlying real bundle one has $p_1(P)=c_1^2-2c_2$, $e(P)=c_2$, and
$w_2(P)=c_1 \mod 2$. Thus \eqref{ppm} implies that
\begin{equation}\label{chern}
   p_1(P_-)=c_1^2-4c_2\,\,\; ,\,\,\;
p_1(P_+)=c_1^2 \,\  \text{ and }\ w_2(P_\pm)\equiv c_1 \mod 2\, .
 \end{equation}
\smallskip

We now discuss the relationship between the associated $\SO(3)$
principal bundles $P_\pm$ and the $U(2)$-reduction $P^*$ and claim
that:
 \begin{equation}\label{chernpm}
   P_-=P^*/Z\,\,\; ,\,\,\;
P_+=(P^*/\SU(2))\times_{\SO(2)}\SO(3)
 \end{equation}
 where $Z$ is the center of $\U(2)$.
 Indeed, if we set $\Gamma=\{\pm (1,1)\}$ when in  $\S^3\times\S^3$  and
 $\Gamma=\{\pm 1\}$ when in  $\S^3$, we get

 \smallskip

$$
 P_-=P/\S^3_-=\left[P^*\times_{\U(2)}\SO(4)\right]/\S^3\times\{e\}
 =\left[P^*\times_{(\S^1\times\S^3)/\Gamma}(\S^3\times\S^3)/\Gamma\right]
 /\S^3\times\{e\}
 $$

 $$
 =P^*\times_{{(\S^1\times\S^3)/\Gamma}}\left[\{e\}\times(\S^3/\Gamma)\right]=P^*/\left[(S^1/\Gamma)\times\{e\}\right]=P^*/Z
 $$

\bigskip

 and for the second bundle

 \smallskip

$$
 P_+=P/\S^3_+=\left[P^*\times_{(\S^1\times\S^3)/\Gamma}(\S^3\times\S^3)/\Gamma\right]
 /\{e\}\times\S^3=P^*\times_{(\S^1\times\S^3)/\Gamma}\left[(\S^3/\Gamma)\times\{e\}\right]
 $$

 $$
 \qquad\qquad =(P^*/\SU(2))\times_{(S^1/\Gamma)\times\{e\}}\left[(\S^3/\Gamma)\times\{e\}\right] =
 (P^*/\SU(2))\times_{\SO(2)}\SO(3)
 $$

 \smallskip

  Thus the structure group of
 $P_+$ reduces to $\SO(2)$. The principal bundle of this reduced bundle is

 \begin{equation}\label{cherncircle}
 \S^1\cong \U(2)/\SU(2) \to P^*/\SU(2)\to
 P^*/\U(2) \ \text{ with Euler class }\ e=c_1(P^*)
 \end{equation}

\no To see that this bundle indeed has Euler class $c_1(P^*)$, we
consider the commutative diagram of classifying spaces:
 \begin{equation*}
\begin{split}
\xymatrix{
P^*\ar[r] \ar[d] & B_{\U(2)} \ar[d]^{\det}\\
P^*/\SU(2)  \ar[r] & B_{\S^1}  }
\end{split}
\end{equation*}
The isomorphism $\U(2)/\SU(2)\cong \S^1$ is induced by the homomorphism
 $\det\colon\U(2)\to \S^1$. Since $\det(\diag(e^{i\theta},e^{i\psi}))=e^{i(\theta+\psi})$,
 the induced map  $B_{\U(2) }\to B_{\S^1}$ in cohomology
  takes $z\to x+y$ in the natural basis of the
 cohomology of the maximal tori, and since $c_1=x+y$, the claim follows.

\smallskip

We now specialize to the case where the base is 4-dimensional.
 Principal $\SO(4)$ bundles $P$ over a compact
simply connected 4-manifold are classified by the characteristic
classes $p_1(P)\ , e(P) $ and $w_2(P)$, and principal $\SO(3)$
bundles  by $p_1(P)$ and $w_2(P)$, see \cite{DW}. But these classes
cannot be assigned arbitrarily. To describe the restriction, we
identify $p_1(P)$ and $ e(P) $ with an integer, using a choice of an
orientation class. For an $\SO(3)$ principal bundle the value of
$w_2(P)$ is arbitrary. The value of $p_1(P)$ on the
other hand satisfies $p_1(P)\equiv e^2 \mod 4$ where $e \in
H^2(M,\Z)$ is the Euler class of a principal circle bundle with $e
\equiv w_2(P) \mod 2$.  Via equation \eqref{ppm} this completely
describes the possible values of the invariants for principal
$\SO(4)$ bundles as well.

\smallskip

In the case of bundles over $\CP^2$, one thus has the following.
Throughout the article, we use  the generator $x$ of $H^2(\CP^2,\Z)$
which is given by the Euler class of the Hopf bundle. The cohomology
class $x^2$ is our choice of an orientation class in
$H^4(\CP^2,\Z)$. The invariants $p_1$ and $e$ are then identified
with integers by evaluation on the fundamental class.  For principal
$\SO(3)$ bundles one then has
\begin{equation}\label{SO3} p_1(P)\equiv 1 \mod 4
\ \ \text{ if } \ \  w_2(P)\ne 0 \ \ \text{ and }\ \   p_1(P)\equiv 0\mod 4 \ \ \text{ if } \ \ w_2(P)= 0.
\end{equation}
 For an
$\SO(4)$ principal bundle $P$ with $w_2(P)=w_2(P_\pm)\ne 0$ one thus has
$p_1(P_-)=4a+1$ and $p_1(P_+)=4b+1$ for some $ a,b\in\Z$ and hence
\begin{equation}\label{SO4nonspin}
p_1(P)=2a+2b+1 \ \ , \ \  e(P)=a-b  \ \ \text{ if }\ \  w_2(P)\ne 0
\end{equation}
  If on the other hand $w_2(P)=0$, one has
$p_1(P_-)=4a$ and $p_1(P_+)=4b$ for some $ a,b\in\Z$ and hence
\begin{equation}\label{SO4spin}
p_1(P)=2a+2b \ \ , \ \  e(P)=a-b  \ \ \text{ if }\ \  w_2(P)= 0.
\end{equation}
 We describe the bundles by
specifying these two (arbitrary) integers $a,b$.

In the case of $\U(2)$ bundles over $\CP^2$, they are classified by
$c_1=r\, x$ and $c_2= s\, x^2$ and $r,s\in\Z$ can be chosen
arbitrarily. The structure group of an $\SO(4)$ bundle reduces to $\U(2)$
if and only if $p_1(P_+)$ is a square, and to $\T^2\subset\SO(4)$ if
and only if both $p_1(P_+)$ and $p_1(P_-)$ are squares.

\smallskip

We will also use the fact that for an $\SO(3)$ principal bundle
$P$ over $\CP^2$ one has
\begin{equation}\label{SO3p1}
|p_1(P)|= |H^4(P,\Z)| \ \text{ and } \pi_1(P)=0 \text{  if } w_2\ne 0
\text{  or }  \pi_1(P)=\Z_2 \text{  if } w_2=0
\end{equation}
 see
\cite[Proposition 3.6]{GZ2}.

\bigskip

\section{Kreck-Stolz invariants}
\label{KS}

 The
Kreck-Stolz invariants are based on the Eells-Kuiper $\mu$-invariant
and are defined as linear combinations of relative characteristic
numbers of appropriate bounding manifolds.  They were introduced and
calculated  for certain homogeneous spaces $L_{a,b}$ in \cite{KS1},
for the Aloff-Wallach spaces in \cite{KS2}, and for
 most of the Eschenburg spaces in \cite{Kr2}.  In the case of smooth
  simply connected closed seven dimensional manifolds $M$ of cohomology type $E_r$ or $\bar{E}_r$ ,
   the Kreck-Stolz invariants
 provide a classification up to homeomorphism and diffeomorphism.

Let $M$ be a  seven dimensional  oriented manifold of cohomology type $E_r$ or $\bar{E}_r$ and
$u\in H^2(M,\Z) \cong \Z$ a generator. If $W$ is an 8 dimensional smooth  manifold bounding $M$, with orientation inducing the orientation of $M$,  and if there exist
    elements $z,c\in H^2(W,\Z)$
such that
\begin{equation}
\begin{aligned}\label{KSzc}
    \partial W&=M  , \ z|M=u  , \ c|M=0  ,\\
     w_2(TW)&=c \!\!\mod 2 \ \text{ if $M^7$ of type $E_r$,} \\
      w_2(TW)&= c  + z \!\!\mod 2 \ \text{ if $M^7$ of type $\bar{E}_r$,}
\end{aligned}
\end{equation}
  one defines characteristic numbers $S_i(W,z,c) \in
\Q\,,\,i=1,2,3,$ as follows.

\begin{equation}\label{KSdef}
\begin{aligned}
S_1(W,z,c) &= <  e^{\frac{c+d}{2}} \cdot \hat A(W),[W,\partial W]> \\
S_2(W,z,c) &= <ch(\lambda(z) -1)  \cdot e^{\frac{c+d}{2}} \cdot \hat A(W),[W,\partial W]> \\
S_3(W,z,c) &= <ch(\lambda^2(z) -1)    \cdot e^{\frac{c+d}{2}} \cdot \hat A(W),[W,\partial W]>\\
 d= 0  &\text{ if $M^7$ is spin, and } d=z \text{ if $M^7$ is  not spin }
  \end{aligned}
  \end{equation}
  Here $\lambda(z)$ stands for the complex line bundle over $W$ with
first Chern class $z$, $ch$ is the Chern character, $ \hat A(W)$ the
$\hat A$ polynomial of $W$, and $[W,\partial W]$  a fundamental
class of $W$ which, restricted to the boundary, is the fundamental class of $M$.
 The integrality of these characteristic numbers for
closed manifolds, see \cite[Theorem 26.1.1]{Hi}, implies
  that $S_i(W,z,c)$ mod $\Z$ depends only on $\partial W=M$,
   and in particular not on the choice of sign for $u,z$ and $c$.
   Notice though that all $S_i$ change sign, if one changes the
   orientation of $M$.  Hence for manifolds of cohomology type $E_r$
 or  $\bar E_r$  one defines:
  $$s_i(M^7) =S_i(W^8,z,c) \mod 1.$$
  The Kreck-Stolz invariants can  be
interpreted as lying in $ \Q/\Z$. In \cite{KS1} it was shown that for
any manifold of cohomology type $E_r$ or $\bar{E}_r$, one can find a
bounding manifold $W$ such that
\eqref{KSzc} is satisfied with $c=0$.
In our examples we will not be able to always find an explicit
bounding manifold $W$ which is spin if $M$ is spin, but we will be able
to find a $W$ and $z,c$  which satisfy \eqref{KSzc}.
 M. Kreck and S. Stolz showed that the invariants $s_i(M^7) $  are diffeomorphism invariants, and
$$\bar s_1(M) = 28
 \, s_1(M) \ \ \text{ and }\ \ \bar s_i(M) =  s_i(M)\,,\,i=2,3$$
 are homeomorphism invariants.

\smallskip

  We now express these invariants explicitly in terms of the
  Pontryagin class $p_1=p_1(TW)  , \  \sign(W)$, and $z,c$. Recall that
 $\hat A(W) = 1 - \frac{1}{2^3 \cdot 3}
\,p_1 + \frac{1}{2^7\cdot 45}\,(-4\,p_2 + 7\,p_1^2)$ and $\sign(W) =
\frac{1}{45}\,(7\,p_2 - p_1^2)$ and hence  $\hat A(W) = 1 -
\frac{1}{2^3 \cdot 3}
 \,p_1 - \frac{1}{2^5\cdot 7}\,\sign(W) + \frac{1}{2^7 \cdot
 7}\,p_1^2\,.$  Furthermore,
 $ch(\lambda(z)) = e^z$ and hence $ch(\lambda(z)-1) = e^z -1$.
Thus we obtain:

\begin{equation}
    M^7 \text{ spin }
\end{equation}

\begin{equation*}\label{KSspin}
\begin{aligned}   S_1(W,c,z) &=  -\frac{1}{2^5 \cdot 7}\,\sign(W) +
                                                                        \frac{1}{2^7\cdot 7}\,p_1^2 - \frac{1}{2^6 \cdot 3}\,c^2\,p_1 + \frac{1}{2^7 \cdot 3} \,c^4\\
S_2(W,c,z) &=    - \frac{1}{2^4 \cdot 3}\,z^2\,p_1    + \frac{1}{2^3 \cdot 3}z^4        -\frac{1}{2^4 \cdot 3}\,z\,c\,p_1 +
                                   \frac{1}{2^4 \cdot 3}\,z\,c^3 + \frac{1}{2^4}\,z^2\,c^2 + \frac{1}{2^2 \cdot 3}\,z^3\,c \\
 S_3(W,c,z) &= - \frac{1}{2^2 \cdot 3}\,z^2\,p_1   + \frac{2}{3}\,z^4 -\frac{1}{2^3\cdot 3}\,z\,c\,p_1 + \frac{1}{2^3\cdot 3}\,z\,c^3
                     + \frac{1}{2^2}\,z^2\,c^2  + \frac{2}{3}\,z^3\,c  .
 \end{aligned}
 \end{equation*}

\bigskip
\begin{equation}
    M^7 \text{ non spin }
\end{equation}

\begin{equation*}\label{KSnonspin}
\begin{aligned}
 S_1(W,c,z) &=  -\frac{1}{2^5 \cdot 7}\,\sign(W) +
\frac{1}{2^7\cdot 7}\,p_1^2 - \frac{1}{2^6 \cdot 3}\,z^2\,p_1 +    \frac{1}{2^7 \cdot 3}\,z^4 - \frac{1}{2^5 \cdot 3} \,z\,c\,p_1   \\
                                            &  - \frac{1}{2^6 \cdot 3}\,c^2\,p_1 + \frac{1}{2^7 \cdot 3}\,c^4   + \frac{1}{2^5 \cdot 3}\,z\,c^3+\frac{1}{2^5 \cdot 3}\,z^3\,c+\frac{1}{2^6}\,z^2\,c^2 \\
S_2(W,c,z) &=    - \frac{1}{2^3 \cdot 3}\,z^2\,p_1    + \frac{5}{2^3\cdot 3}z^4        -\frac{1}{2^4 \cdot 3}\,z\,c\,p_1 + \frac{1}{2^4 \cdot 3}\,z\,c^3 +
                                \frac{1}{2^3}\,z^2\,c^2 + \frac{13}{2^4 \cdot 3}\,z^3\,c \\
S_3(W,c,z) &= - \frac{1}{2^3}\,z^2\,p_1   + \frac{13}{2^3}\,z^4 -\frac{1}{2^3\cdot 3}\,z\,c\,p_1 + \frac{1}{2^3\cdot 3}\,z\,c^3
                     + \frac{3}{2^3}\,z^2\,c^2  + \frac{31}{3\cdot 2^3}\,z^3\,c  .
 \end{aligned}
  \end{equation*}

  \bigskip

These formulas need to be interpreted as follows.  Since $\partial
W=M$, we have $H^3(\partial W,\Q)=H^4(\partial W,\Q)=0$  and hence
the inclusion $j\colon (W,\emptyset )\to (W,\partial W)$ induces an
isomorphism $j^*\colon H^4(W,\partial W,\Q)
 \to   H^4(W,\Q)$. Thus the characteristic classes $p_1,z^2,c^2,zc$
 in $H^4(W,\Q)$ can be pulled back to relative classes in $H^4(W,\partial
 W,\Q)$ and
 the classes
$p_1^2, z^2\,p_1,z^4,$  etc. in the above formulas are abbreviations for the characteristic numbers
$$ p_1^2=\langle (j^*)^{-1}(p_1) \cup
p_1,[W,\partial W]\rangle \ , \  z^3c=\langle (j^*)^{-1}(z^2) \cup
zc,[W,\partial W]\rangle ,\  \text{etc.}.$$

\smallskip

  The main classification theorem in \cite{KS2},Theorem 3.1, can now be stated
as follows:

\begin{thm}[Kreck-Stolz]\label{KSdiffeo}
Two simply connected smooth manifolds $M_1 , M_2$ which are both of type $E_r$, or  both of type $\bar{E}_r$,
are orientation preserving diffeomorphic (homeomorphic) if and only
if  $s_i(M_1)=s_i(M_2)$ (resp. $\bar s_i(M_1)=\bar s_i(M_2)$) for
$\,i=1,2,3$.
\end{thm}

\noindent
For orientation reversing diffeomorphisms one changes the signs of  the invariants.
\smallskip

Recall that $E_r$ and $\bar E_r$ fall into two homotopy types
depending on whether $\pi_4(M)=0$ or $\Z_2$.  For some of these
manifolds B. Kruggel obtained a homotopy classification, see
 \cite[Theorem 0.1]{Kr2}, \cite[Theorem 3.4]{Kr1}:
 \newpage
\begin{thm}[Kruggel]\label{Ktop} For simply connected smooth
manifolds  $M_1 , M_2$ one has:
\begin{itemize}
\item[a)]  If $M_i$ are both of type $E_r$ and $\pi_4(M_i)  = 0$,  then $r$ is odd, and $M_1$ and $M_2$
are orientation preserving homotopy equivalent if and only if $lk(M_1) \equiv lk(M_2) \in \Z_r$  and $2\,r\,s_2(M_1) \equiv  2\,r\,s_2(M_2) \in \Q/\Z$.
\item[b)]   If $M_i$ are both of type $E_r$ and $\pi_4(M_i)  \cong \Z_2$ and if $r$ is odd, $M_1$ and $M_2$
are orientation preserving homotopy equivalent if and only if  $lk(M_1) \equiv lk(M_2) \in \Z_r$ and $r\,s_2(M_1) \equiv  r\,s_2(M_2)  \in \Q/\Z$.
\item[(c)] If $M_i$ are both of type $\bar E_r$ with $r$ divisible by $24$,  $M_1$ and $M_2$ are orientation
preserving homotopy equivalent if and only if  $lk(M_1) \equiv lk(M_2) \in \Z_r$ and $p_1(M_1) \equiv p_1(M_2) \mod 24$.
\end{itemize}
\end{thm}
\no In the remaining cases, the homotopy classification has not yet
been finished.

P. Montagantirud showed in \cite{Mo} that in the case of manifolds of
type $\bar E_r$, or of type $E_r$ with $r$ odd, one can replace
$s_3$ with the linking form $lk$ in the homeomorphism and
diffeomorphism classification. In other words, $28 \,s_1,s_2$ and $lk$
classify the manifold up to homeomorphism, and $s_1,s_2$ and $lk$ up
to diffeomorphism.  Furthermore, in the case of manifolds of type $E_r$ with $r$ odd,
he proves that one can replace the invariant $28\,s_1$ with $p_1$.

\bigskip

\section{Topology of sphere bundles over $\CP^2$}
\label{3sphere}

\bigskip

 We start with the family of $\Sph^3$-bundles over
$\CP^2$. As we will see, the total space of such a bundle is spin if
and only if the bundle itself is not spin.  Our goal is to see when
the total space of such a bundle is diffeomorphic to an
Eschenburg space,  and we thus   restrict ourselves in this section to sphere
bundles which are not spin. According to \eqref{SO4nonspin}, they are classified by two integers $a,b$
and we let
 $$
\Sph^3\to S_{a,\, b }\overset{\pi}{\longrightarrow} \CP^2
 \quad \text{ with } \quad p_1(S_{a,\, b })=(2a+2b+1)\,x^2 ,\ e(S_{a,\, b })=(a-b)\,x^2 \text{ and } w_2\ne 0
$$
where $x$ is a generator of $H^2(\CP^2;\Z)$. A change of orientation
corresponds to changing the sign of $e$ but not of $p_1$. Thus
$S_{a,\, b }$ and $S_{b,\, a }$ are orientation reversing
diffeomorphic.

\smallskip

It turns out  that it is also the non-spin bundles which are known
to admit non-negative curvature, as was shown in \cite{GZ2}:

\begin{thm}[Grove-Ziller]
Every sphere bundle over $\CP^2$ with $w_2\ne 0$  admits a metric
with non-negative sectional curvature.
\end{thm}

\smallskip

  As for their topology we have:

\begin{prop}\label{3top}
The manifolds $S_{a,\, b }$ have cohomology type $E_r$ with
$r=|a-b|$, as long as $a\ne b$. Their
first Pontryagin class is given by $p_1(TS_{a,\, b })\equiv (2\,a+2\,b+4) \text{ mod }
r$.
\end{prop}
\begin{proof}
The cohomology ring structure and  $r=|a-b|$ immediately follows
from the Gysin sequence. In particular, $H^2(S_{a,\, b }\;
,\Z)\cong\Z$ with $u=\pi^*(x)$  a generator.

For the manifolds $S=S_{a,\, b }$ we have $TS \cong \pi^*(T\CP^2)
\oplus V$ where $V \oplus \Id = \pi^*E$ with $\Id$ a trivial bundle
and $E$ the vector bundle associated to the sphere bundle $S$. Hence
$p_1(V) = \pi^*(p_1(E))$ and, since $p_1(T\CP^2)=3\,x^2$ and $p_1(S)
= (2a+2b+1)\,x^2$, we have $p_1(TS)=(2\,a+2\,b+4)\,u^2 $.  Also, $S$
is spin since $w_2(TS) = \pi^*(w_2(T\CP^2) + w_2(E)) =0$.

\end{proof}

\begin{rem*}
In the case of $a=b$ the bundle has the same cohomology ring as $\Sph^3\times \CP^2$. Such manifolds
are not classified by the Kreck-Stolz invariants.  However, they are also not homotopy equivalent to any
known example of positive curvature.
\end{rem*}

 We now compute the Kreck-Stolz
invariants and for this purpose need to fix the orientation. In the fibration of
$S_{a,b}$, the base $\CP^2$ is oriented
 via $x^2\in H^4(\CP^2)$ and the orientation of the fiber is determined by the sign
 of the Euler class. This determines the orientation on the total space.

\begin{prop}\label{KSsphere}
The Kreck-Stolz invariants for $S_{a,\, b }$ with $a\ne b$ are
given by:
$$
\begin{aligned}
s_1(S_{a,\, b }) &  \equiv \frac{1}{2^5 \cdot 7 \cdot (a-b)}\,(a+b+2)^2
-\frac{\signs(a-b)}{2^5\cdot 7}
\mod 1 \\
s_2(S_{a,\, b }) &  \equiv  \frac{-1}{2^3 \cdot 3 \cdot (a-b)}\,(a+b+1) \mod 1\\
s_3(S_{a,\, b}) &  \equiv  \frac{-1}{2 \cdot 3 \cdot (a-b)}\,(a+b-2) \mod 1
 \end{aligned}
 $$
\end{prop}
\begin{proof} We will use the notation
established in the proof of \eqref{3top} and set $u=\pi^*(x) $.  Let
$E$ be  the vector bundle associated to the sphere bundle $S_{a,\, b
}$ and $\bar\pi\colon W_{a,\, b } \to \CP^2$ its  disk bundle. Hence
$W=W_{a,\, b }$ is a natural choice for a bounding manifold and we
identify the cohomology of $W$ with that of $\CP^2$. Thus
$z:=\bar\pi^*(x)$ is a generator for $H^2(W;\Z)\cong \Z$. Since the
restriction of $\bar\pi$ to $\partial W=S_{a,\, b }$ is $\pi$, it
follows that $z|\partial W =\pi^*(x) =u$. For the tangent bundle of
$W$
 we have $TW =\bar\pi^*( T\CP^2) \oplus \bar\pi^* E$ and hence  $w_2(TW) = \bar\pi^*( w_2(T\CP^2) +
 w_2(E) )= 0$. Thus $S_{a,b}$ and $W$ are both spin,  and we can choose $c=0$ in \eqref{KSspin}. Furthermore,
  $p_1(TW) =\bar\pi^*( p_1(T\CP^2) + p_1(E)) = 3\,x^2 + (2\,a+2\,b+1)\,x^2
 =
(2\,a+2\,b+4)\,z^2 \in H^4(W;\Z) \cong \Z\,.$

The sphere bundle and hence the  disk bundle are assumed to be
oriented and we let $U\in H^4(W,\partial W)\cong\Z$ be the
corresponding Thom class. The orientation on $\CP^2$ is defined by
$\langle x^2 , [\CP^2]\rangle =1$ and we define the orientation on
$W$ such that $U\cap [W,\partial W]=[\CP^2]$. Thus $\langle U \cup
z^2,[W,\partial W]\rangle =\langle z^2,U\cap [W,\partial W]\rangle=
\langle x^2,[\CP^2]\rangle =1 $.   If $j\colon W\to (W,\partial W)$
is the inclusion, $j^*(U)$ is the Euler class and hence
$j^*(U)=(a-b)z^2$. Thus  we obtain
$$ \begin{aligned}
\langle(j^{-1})^*(z^2) \cup z^2,[W,\partial W]\rangle
                                          &=\frac{1}{a-b}\langle U
\cup z^2,[W,\partial W]\rangle = \frac{1}{a-b}\\
 \langle(j^{-1})^*(p_1) \cup p_1,[W,\partial W]\rangle
 &= \frac{1}{a-b}\langle\,(2\,a+2\,b+4)^2 \cdot U
\cup z^2,[W,\partial W]\rangle\\
&= \frac{1}{a-b}\,(2\,a+2\,b+4)^2\\
 \langle(j^{-1})^*(p_1) \cup z^2,[W,\partial W]\rangle
 &=\frac{1}{a-b} \langle\,(2\,a+2\,b+4) \cdot U
\cup z^2,[W,\partial W]\rangle\\
 &= \frac{1}{a-b}\,(2\,a+2\,b+4)\,.
\end{aligned}$$

Recall that the signature of a manifold with boundary is defined as
the signature of the quadratic form on $H^4(W,\partial W)$ given by
$v \mapsto \langle j^*(v) \cup v,[W, \partial W]\rangle$. Since
$H^4(W,\partial W)$ is generated by $U$ and $\langle (j^*)(U) \cup
U,[W,\partial W]\rangle=a-b$
 we have $\sign(W) =
\signs(a-b)$.
 Substituting into \eqref{KSspin} proves our claim.
\end{proof}

\begin{rem*}
Notice that $s_3\equiv 4 \,s_2+\frac{1}{2r} \mod 1$ and thus for
(orientation preserving diffeomorphisms) $s_2$ determines $s_3$.
\end{rem*}

If $r>1$ we also have the linking form:

\begin{cor}\label{linkingS3}
The linking form of $S_{a,\, b }$ with  $a\ne b$ is standard, i.e.,
$lk(S_{a,\; b}) \equiv \dfrac{1}{(a-b)} \in \Q/\Z$.
\end{cor}
\begin{proof}
As discussed in Section 1, the linking form is a bilinear form $L:
H^4(M;\Z) \times H^4(M;\Z) \longrightarrow \Q/\Z$ and is completely
determined by $L(u^2,u^2)$. It turns out by \cite{KS2}, see also
\cite{Mo}, that for manifolds of type $E_r$ we can also express the
linking form as the characteristic number $z^4$.  As seen above in
the case of $S_{a,\; b}$ we obtain $z^4 = \frac{1} {a-b}$ and hence
$L(u^2,u^2) \equiv  \frac{1}{a-b}  \in \Q/\Z$.
\end{proof}

\smallskip

As a consequence of \pref{KSsphere}, together with \tref{KSdiffeo},
one easily obtains a homeomorphism and diffeomorphism classification
of the manifolds  $S_{a,b}$:

\begin{cor}\label{hdclass} The manifolds $S_{a,b}$ and $S_{a',b'}$ with $r=a-b=a'-b'>0$ are
\begin{itemize}
\item[(a)]
 orientation preserving  homeomorphic  if and only if  $a\equiv a' \mod 12\,
 r$.
 \item[(b)] orientation preserving diffeomorphic  if and
only if $$a\equiv a' \mod 12\, r  \ \text{ and } \
 (a-a')\,[a+a'-r+2] \equiv  0 \mod 2^3\cdot 7\, \cdot r.$$
\item[(a')]  orientation
reversing  homeomorphic  if and only if  $r=1$ and $a\equiv -a' \mod 12\,.$
 \item[(b')]  orientation  reversing diffeomorphic  if and
only if  $r=1$ and $a(a+1)\equiv -a' (a'+1) \mod 2^3\cdot 7\,.$
\end{itemize}
\end{cor}

\begin{rem*} (a)
Notice that if $a\equiv a' \mod 168\, r$, then
$S_{a,r-a}$ and $S_{a',r-a'}$ are diffeomorphic. Thus each sphere
bundle is diffeomorphic to infinitely many other sphere bundles.

(b) As far as the homotopy type is concerned, we note that
$\pi_4(S_{a,b}) \cong \Z_2$ if $r=|a-b|$ is even, as follows from
\tref{Ktop} (a).  We suspect that if
 $r$ is odd,  $\pi_4(S_{a,b})=0$. For example, \cref{sameAW}
implies that $\pi_4(S_{\pm 1, p(p+1)})=0$.

In any case, if $\pi_4(S_{a,b})=0$, $S_{a,b}$ and $S_{a',b'}$ with
$a-b=a'-b'>0$ are orientation preserving   homotopy equivalent if
and only if $a\equiv a' \mod 6$, and
 orientation  reversing  homotopy equivalent if and only if $a+a'\equiv r-1 \mod 6$.
If $\pi_4(S_{a,b}) \cong \Z_2$, the same holds  mod 12.
\end{rem*}

\bigskip

\section{Topology of Circle Bundles}
\label{circle}

\bigskip

In this section we discuss the manifolds $M^t_{a,\, b }$ described
in the introduction. We start with  6-dimensional manifolds  which
are $\Sph^2$ bundles over $\CP^2$. As we will see, it is again important
to assume that they are not spin. According to \eqref{SO3}, the
corresponding $\SO(3)$ principal bundle  satisfies $p_1\equiv 1 \mod
4$ and we define:
$$
\Sph^2\to N_t\overset{\pi}{\longrightarrow} \CP^2
 \quad \text{ with } \quad p_1(N_t)=(1-4t)x^2 \text{ and } w_2\ne 0
$$
for some integer $t$.

\smallskip

In order to compute the cohomology ring of $N_t$ we regard the
2-sphere bundle as the projectivization of a rank 2 complex vector
bundle. For this purpose, let $\U(2)\to P \to \CP^2$ be a principal
$\U(2)$ bundle with associated vector bundle
$E=P\times_{\U(2)}\C^2$. Such bundles are classified by their Chern classes
$c_1$ and $ c_2$ and we define
$$\U(2)\to P_t \overset{\tau}{\longrightarrow} \CP^2
 \quad \text{ with } \quad c_1(P_t)=x \text{ and } c_2(P_t)=tx^2
$$
 Since $c_1 $ mod
$2 = w_2$, such a bundle is not spin. If $Z=\diag(z,z)\subset\U(2)$
denotes the center of $\U(2)$, we obtain an $\SO(3)$ principal bundle
$$\U(2)/Z\cong \SO(3)\to Q_t:=P_t/Z\to \CP^2.$$
According to \eqref{chernpm}, we have $P_t/Z=P_-$ and hence, by \eqref{chern},
$p_1(Q_t)=c_1^2-4c_2=(1-4t)x^2$.
Furthermore, $w_2(Q_t)=w_2(P_-)=w_2(P)\ne 0$. Thus
$$\SO(3)\to Q_t\to \CP^2 \quad \text{ with } \quad p_1(Q_t)=(1-4t)x^2 \text{ and } w_2(Q_t)\ne 0$$
 is the $\SO(3)$ principal bundle associated to the 2-sphere bundle $N_t$.
This implies that
$$P(E)\simeq P_t\times_{\U(2)}\CP^1\simeq (P_t/Z)\times_{\SO(3)}\Sph^2\simeq Q_t\times_{\SO(3)}\Sph^2\simeq N_t
$$
i.e., we can  regard $N_t$ as the projectivization of $E$. Furthermore,  $N_t=P_t/\T^2$ since
$$N_t\simeq P(E)\simeq P_t\times_{\U(2)}\CP^1\simeq
P_t\times_{\U(2)}\U(2)/\T^2\simeq P_t/\T^2.$$

\bigskip

We can now apply Leray-Hirsch to compute  the cohomology ring of $P(E)=N_t$:
\begin{equation}
\begin{aligned}
H^{*}(P(E)) &\cong H^{*}(\CP^2)[y]/[y^2 +c_1(E)y + c_2(E)]\\
                     &\cong H^{*}(\CP^2)[y]/[y^2 +x\,y +t\,x^2]\\
                     &\cong \Z[x,y]/[x^3=0,y^2+x\,y+t\,x^2=0]
\end{aligned}
\end{equation}
Here we have identified $x\in H^2(\CP^2)$ with
$\pi^*(x)\in\H^2(N_t)\cong  H^2(P(E))$. Furthermore, the generator
$y$ is defined by $y=c_1(S^*)$ where $S^*$ is the dual of the
tautological complex line bundle $S$ over $P(E)$. Hence
\begin{equation}\label{cohbase}
\begin{aligned}
H^2(N_t) &\cong \Z \oplus \Z \,\text{with generators}\, x,y\,;\\
H^4(N_t) &\cong  \Z \oplus \Z \,\text{with generators}\,
x^2,xy,\,\text{and relationship}\, y^2 =
 -xy-t\,x^2\,;\\
H^6(N_t) &\cong  \Z \,\text{with generator}\, x^2y, \,\text{and}\,
x^3=0,y^2\,x = -x^2y,y^3=
(1-t)\,x^2y\,.\\
\end{aligned}
\end{equation}

Notice that, since the quadratic relationship $y^2 +x\,y +t\,x^2$
has discriminant $ t-\frac{1}{4}$, the manifolds
 $N_t$ all have different homotopy
type.

\bigskip

We now consider circle bundles over $N_t$. They are classified by
their Euler class $e\in H^2(N_t,\Z)$. For symmetry reasons, see
\cref{symm}, we define
$$\Sph^1\to  M^t_{a, b}
 \overset{\sigma}{\longrightarrow}  N_t
 \quad \text{ with }\quad  e(M^t_{a, b} )=ax +(a+b)y
$$
where $a,\, b $ are arbitrary integers. In order to ensure that
$M^t_{a,\, b }$ is simply connected we assume that $(a,\, b ) =1$. The total
space is oriented via the orientation class $x^2y\in H^6(N_t)$ on the base
and the orientation on the fiber given by $e$. Thus $M^t_{a, b}$ and $M^t_{-a, -b}$ are orientation reversing diffeomorphic.

\smallskip

For the basic  topological invariants of $M^t_{a,\, b }$ we obtain:

\begin{prop}\label{circletop}
The manifolds $M^t_{a,\, b }$ have cohomology type $E_r$ with
$r=|t(a+b)^2-ab|$,  as long as  $t(a^2+b^2)\ne ab$. Furthermore, the
first Pontryagin class is given by $p_1(TM^t_{a,\, b })\equiv
4(1-t)(a+b)^2 \text{ mod } r$.
\end{prop}
\begin{proof}
The Gysin sequence for the bundle $\, M^t_{a,\, b }\to N_t \,$  yields
$\,H^2(M^t_{a,\, b };\Z)
 \cong H^5(M^t_{a,\, b };\Z) \cong \Z; H^1(M^t_{a,\, b };\Z) \cong H^3(M^t_{a,\, b };\Z) \cong H^6(M^t_{a,\, b };\Z)
 \cong 0\,$ and $H^4(M^t_{a,\, b };\Z) \cong \Z_r$ where $\Z_r$ is the cokernel of
$$0 \longrightarrow H^2(N_t;\Z)\stackrel{\cup e}{\longrightarrow}   H^4(N_t;\Z) \longrightarrow
 H^4(M^t_{a,\, b };\Z)\to 0 \,.$$
Since the cup product with the Euler class sends $x$ to $a\,x^2 +
(a+b)\, xy$ and $y$ to $ axy+(a+b)y^2=-t\,(a+ b) x^2\, -b\, xy$ we obtain for this
homomorphism that
$$
 \det\begin{pmatrix} a& - t(a+b)  \\
 a+b &  -b
 \end{pmatrix} = t(a+b)^2-ab
$$
and thus the cokernel has order $|t(a+b)^2-ab|$ as long as $t(a+b)^2-ab\ne 0 $.  One easily sees
that the cokernel is cyclic since $(a,\;
b)=1$, and hence $H^4(M^t_{a,\, b };\Z) \cong
\Z_r$ with $r =|t(a+b)^2-ab|$.

To compute the characteristic classes we observe that the tangent
bundles split:  $TM^t_{a,\, b }=\sigma^*(TN_t)\oplus\Id$ and  $TN_t
= \pi^*(T\CP^2) \oplus V$ with $V \oplus \Id = \pi^*(
 E^3)$, and where $E^3$ is the rank 3 vector bundle corresponding to the 2-sphere bundle $N_t$.
   Hence $p_1(TM^t_{a,\, b }) = \sigma^*\pi^*( 3x^2+ (1 - 4t)\,x^2)$.
   Applying
   the Gysin sequence again, we see that $0=\sigma^*(e)=
\sigma^*(ax+(a+b)y)$ and that $\sigma^*\colon H^2(N_t;\Z)\to H^2(M^t_{a,\, b };\Z)$ is onto. Since $(a,\, b )=1$,
this implies that there exists a generator
$u\in H^2(M^t_{a,\, b };\Z)\cong\Z$ with $\sigma^*(x)=-(a+b)u$ and
$\sigma^*(y)=au$. Hence $p_1(TM^t_{a,\, b })=4(1-t)(a+b)^2u^2$
and similarly  $M^t_{a,\, b }$ is spin since $w_2(TM^t_{a,\;
b})=\sigma^*(w_2(TN_t))$ and $w_2(TN_t)=\pi^*(2x)=0$.
\end{proof}

The Kreck-Stolz invariants for the manifolds $M^t_{a,\, b }$ were
computed in the special cases $t=\pm 1$ in \cite{KS1} and
\cite{AMP1}. For the general case below,  we will use  $s =t(a+b)^2-ab$,
$r=|s|$, since $s$ is not always positive.

\begin{prop}\label{KSCircle}
If $s =t(a+b)^2-ab\ne 0$, the Kreck-Stolz invariants for $M^t_{a,\, b }$   are given
by:
$$
\begin{aligned}
s_1(M^t_{a,\, b })   \equiv\,  & \frac{-1}{2^5 \cdot 7}\,\sign(W) -\frac{(a+b)(t-1)^2}{2^3 \cdot 7\cdot s} \\
                          & \qquad\qquad\qquad +\frac{a+b}{2^5\cdot 3 \cdot 7}\,\{3a\, b+(t-1)\,(8 + (a+b)^2) \}\ \ \  \mod 1  \,,
\end{aligned}
$$
$$
\begin{aligned}
 s_2(M^t_{a,\, b }&)   \equiv\, \frac{1}{2^3\cdot
3}\, \left\{  (t-1)\,(m+n)\,[2 - (a+b)\,(m+n) -2\,(m+n)^2]    \right.             \\
 &\qquad \qquad \quad \left. - a\,m(m+2\,n) - b\,n(n + 2\,m) - 6\,m\,n\,(m+n)   \right\}     \\
                          &+ \frac{1}{2^3\cdot
3\cdot s}\,\left\{     (t^2-1)  (a+b)(n+m)^2 (2-(n+m)^2)  \right.            \\
                          &+(t-1)\left[m^4(3a+b)+n^4(a+3b)-2(a+b)(m+n)^2+2(am^2+bn^2)(2nm-1)\right]        \\
   & \left.\qquad \qquad +am^4+bn^4-6m^2n^2(a+b) -4mn(an^2+bm^2)    \right\} \ \   \mod 1 \,,
 \end{aligned}
$$
$$
\begin{aligned}
s_3(M^t_{a,\, b }&)   \equiv\, \frac{1}{2\cdot
3}\,\left\{  (t-1)\,(m+n)\,[1 - (a+b)\,(m+n) -4\,(m+n)^2] \right.               \\
 &\qquad \qquad \ \ \left. - a\,m(m+2\,n) - b\,n(n + 2\,m)    \right\}  \\
                          &+ \frac{1}{
 3\cdot s}\,\left\{     (t^2-1)  (a+b)(n+m)^2 (1-2(n+m)^2)  \right.            \\
                          &+(t-1)[2m^4(3a+b)+2n^4(a+3b)-(a+b)(m+n)^2+(am^2+bn^2)(8nm-1)]        \\
   & \left. \qquad \ \  \quad +2am^4+2bn^4-12m^2n^2(a+b) -8mn(an^2+bm^2)    \right\}   \mod 1 \,,
\end{aligned}
$$
where $n,m \in \Z$ are chosen such that  $a\,m - b\,n =
 1$. Furthermore,
$$ \sign(W)=\left\{
   \begin{array}{lll}
     &0, &\ \text{ if }\ s>0 \\
     &2, &\ \text{ if }\ s<0 \ \text{ and } \  b+(1-t)(a+b)>0 \\
     -&2, & \ \text{ if }\  s<0 \ \text{ and } \ b+(1-t)(a+b)<0.
   \end{array}
 \right.
$$
\end{prop}

\begin{proof}
 A natural choice
 for a bounding manifold is the  disk bundle $\gs'\colon
 W^8_{a,\, b } \longrightarrow N_t$ of the rank 2 vector bundle $E^2$ associated
 to the circle bundle $\gs$.
 Recall  from the proof of \pref{circletop} that there
exists a generator $u\in H^2(M^t_{a,\, b };\Z)\cong\Z$ with
$\sigma^*(x)=-(a+b)u$ and $\sigma^*(y)=au$. Furthermore, we
identified the cohomology of $W$ with that of $N_t$ via $\sigma'^*$.
In order to apply \eqref{KSspin}, we need to choose classes $z,c \in
H^2(W,\Z)\cong\Z\oplus\Z$ with $z{|\partial W} = u,$  $c|\partial W
= 0$ and $w_2(W) \equiv c \mod 2\,.$ If we choose integers $m,n$
with $a\,m - b\,n = 1$ and set $z = n\,x + (m+n)\,y$ and
$c=e(M^t_{a,b}) = a\,x + (a+b)\,y$, it follows that $\gs^*(z)=u,\,
\sigma^*(c)=0$. Furthermore, $w_2(TW) = \gs'^*(w_2(TN_t)) +
w_2(E^2)= e \mod 2$. Thus $z$ and $c$ have the
 properties required in \eqref{KSzc}. Using $p_1(TW) = \gs'^*(p_1(TN_t)) + p_1(E^2) = 4\,(1-t)\,x^2 +e^2$ and replacing the value of $p_1$ in
\eqref{KSspin} we obtain
$$\begin{aligned}
  S_1(W,c,z) &= -\frac{1}{2^5 \cdot 7}\,\sign(W)+ \frac{1}{2^5\cdot
3\cdot 7}\,(12\,(t-1)^2\,x^4 +
                     8\,(t-1)\,x^2\,c^2-c^4) ,
                     \\
S_2(W,c,z) &=\frac{1}{2^3\cdot 3}\,( 2\,(t-1)\,z\,c\,x^2 +
2\,(t-1)\,z^2\,x^2 + z^2\,c^2 + 2\,z^3\,c
                    + z^4 ),
                      \\
 S_3(W,c,z)
                     &=\frac{1}{2\cdot 3}\,((t-1)\,z\,c\,x^2 + 2\,(t-1)\,z^2\,x^2 + z^2\,c^2 + 4\,z^3\,c
                    + 4\,z^4 ).
  \end{aligned}$$

Recall that the orientation for $N_t$ is chosen so that $\langle x^2\,y, [N_t]\rangle = 1$.
The orientation for the vector bundle $E^2$ defines a Thom class
 $U \in H^2(W,\partial W)\cong\Z$ and we define the orientation on $W$ such that
$U\cap [W,
 \partial W]= [N_t]$.

  We first compute the characteristic numbers involving $c$ by using
   the fact that $j^*(U) = e(E^2)=c$ where $j: W \longrightarrow (W,\partial W)$:
$$\begin{aligned}
c^4 &= \langle c^3 \cup (j^*)^{-1}(c),[W,\partial W]\rangle=\langle c^3 \cup U,[W,\partial W]\rangle\\
       &= \langle(a\,x + (a+b)\,y)^3,[N_t]\rangle
       = a^3+b^3-t(a+b)^3 \\
z^2\, c^2 &= \langle c\,z^2 \cup U,[W,\partial W]\rangle
                = \langle(a\,x + (a+b)\,y) \, (n\,x + (m+n)\,y)^2,[N_t]\rangle \\
                &=  a\, n^2+b\, m^2 -t(a+b)(m+n)^2\\
z^3\, c &= \langle z^3 \cup U, [W,\partial W]\rangle
           = \langle(n\,x + (m+n)\,y)^3, [N_t]\rangle\\
           &= m^3+n^3 -t (m+n)^3\\
x^2\,c^2 &= \langle x^2\,c \cup U,[W,\partial W]\rangle
                = \langle x^2\,(a\,x + (a+b)\,y), [N_t]\rangle = a+b \\
z\,c\, x^2 &= \langle z\,x^2 \cup U,[W,\partial W]\rangle
                = \langle(n\,x + (m+n)\,y)\,x^2, [N_t]\rangle =m+n
\end{aligned}$$
A calculation shows that $s\, x^2 = - (b \,x  +(a+b) \,y) \cup c$ and hence
$$x^4 =- \frac 1 s \langle x^2 \,  (b \,x  +(a+b) \,y) , [N_t]\rangle = -
\frac{a+b}{s}$$
Similarly, $ z^2 =
(\alpha \,x + \beta \,y) \cup c$, where $s\, \alpha = at(n^2-m^2)+2btn(m+n)-n^2b  $ and $s\, \beta =
 t(a+b)(m+n)^2-a\, m^2-b\, n^2$.

 Thus
$$\begin{aligned}
z^4&= \langle z^2 \cup  (\alpha \,x + \beta \,y)\cup U , [W,\partial W]\rangle
=\langle (n\,x + (m+n)\,y)^2 \cup  (\alpha \,x + \beta \,y) , [N_t]\rangle\\
      &= \frac{1}{s}\,\{ -t^2(a+b)(m+n)^4+t[m^4(3a+b)+n^4(a+3b) +4nm(a\, m^2+b\, n^2)]\\
      &\qquad\qquad  -am^4-bn^4\}\\
 z^2\,x^2 &= \frac{1}{s}\,\langle(\alpha \,x + \beta \,y)\,x^2 , [N_t]\rangle =
                 \frac{1}{s}\,( t(a+b)(m+n)^2-a\, m^2-b\, n^2)
 \end{aligned}$$

\smallskip

We now compute the signature form $v\to \langle j^*(v) \cup v,[W,
\partial W]\rangle$ on $H^4(W,\partial W)$.
Since $x,y$ are a basis of $H^2(W)\cong H^2(N_t)$, the classes $x
\cup U, y \cup U$ are a basis of $H^4(W,\partial W)$. Using $j^*(x
\cup U)=x\cup j^*(U) = x\cup c = a\,x^2 + (a+b)\,xy$ and similarly
$j^*(y \cup U) = -t\,(a+b)\,x^2 -b\,xy$, we have
$$\begin{aligned}
\langle j^*(x\cup U) \cup x \cup U,[W,\partial W]\rangle &= \langle(a\,x^2 + (a+b)\,xy) \cup x,[N_t]\rangle  = a+b\\
\langle j^*(x\cup U) \cup y \cup U,[W,\partial W]\rangle &= \langle(a\,x^2 + (a+b)\,xy) \cup y, [N_t]\rangle = -b\\
\langle j^*(y\cup U) \cup x \cup U,[W,\partial W]\rangle &= \langle( -t\,(a+b)\,x^2 -b\,xy) \cup x, [N_t]\rangle = -b\\
\langle j^*(y\cup U) \cup y \cup U,[W,\partial W]\rangle &= \langle( -t\,(a+b)\,x^2 -b\,xy) \cup y, [N_t]\rangle = -t\,(a+b)+b\\
\end{aligned}$$
If we denote by $R$ the matrix of this signature form, one easily sees that $\det R=-s$ and since $\tr R=b+(1-t)(a+b)$, the signature is as claimed.

\smallskip

Combining all of the above, our claim follows from
\eqref{KSspin}.
 \end{proof}

\begin{rem*}
Notice that one can always choose $m$ or $n$ to be divisible by $4$, which
easily implies that the term $(t-1)\,(m+n)\,[2  -2\,(m+n)^2]    -
6\,m\,n\,(m+n)$ vanishes in $s_2$ and the term $(t-1)\,(m+n)\,[1
-4\,(m+n)^2]  \equiv 3 \mod 6$ in $s_3$. Thus a change $(a,b,m,n)\to
(b,a,-n,-m)$ gives the same Kreck-Stolz invariants, confirming the
orientation preserving diffeomorphism in \cref{symm}. Notice also
that a change $(a,b,m,n)\to (-a,-b,-m,-n)$ gives opposite
Kreck-Stolz invariants confirming that $M^t_{a,\, b }$ is
orientation reversing diffeomorphic to $M^t_{-a,\, -b }$.
\end{rem*}

Recall that for manifolds of type $E_r$ with $r>1$  the linking form is equal
to the characteristic number $z^4$ and we hence obtain:

\begin{cor}\label{linkingS1}
If $s =t(a+b)^2-ab$ and $r=|s|>1$, the linking form of $M^t_{a,\, b }$ is given by
$$lk(M^t_{a,\, b })=\frac{1}{s}\,\{ -t^2(a+b)(m+n)^4+t[m^4(3a+b)+n^4(a+3b)
 +4nm(a\, m^2+b\, n^2)]  -am^4-bn^4\} $$
      in $\Q/\Z$, where   $n,m \in \Z$ are chosen such that  $a\,m - b\,n =
 1$.
\end{cor}

\begin{rem*}
As far as the homotopy invariant $\pi_4(M^t_{a,b})$ is concerned, we
notice that it does not depend on $a,b$ as follows from the circle
bundle $\S^1\to P_t\to M^t_{a,b}$. This implies that
$\pi_4(M^t_{a,b})=\Z_2$ for $t$ even. Indeed, if $t$ even, $a$ even
and $b$ odd, the order $r=|t(a+b)^2-ab|$ is  even and the claim
follows from \tref{Ktop} (a). If $t$ is odd, we suspect that
$\pi_4(M^t_{a,b})=0$. For example, from \pref{t=pm 1} it follows
that $\pi_4(M^{\pm 1}_{a,b})=0$.
\end{rem*}

\bigskip

\section{Geometry  of  sphere bundles and circle Bundles}
\label{geomspherecircle}

\bigskip

In this section we study the geometry of the sphere and circle bundles defined
in Section 3 and 4 and the relationships between them.

We remind the reader of the various bundle structures in Section 3 and 4 in the following diagram:

\smallskip

\begin{equation*}
\begin{split}
\xymatrix{
\S^1 \ar[d] & \S^1 \ar[d] \\
\U(2) \ar[r] \ar[d]^{/Z}& P_t \ar[r]^\tau \ar[d]^{/Z} &
\CP^2 \ \\
\SO(3) \ar[r] & Q_t \ar[r]  &
\CP^2 \ \\
\CP^1 \ar[r] & N_t \ar[r]^\pi  &
\CP^2  }\qquad\qquad
\xymatrix{
\T^2 \ar[rd] & \S^1 \ar[d] &\S^1\ar[d]\\
\S^1 \ar[r] & P_t \ar[r] \ar[d] \ar[rd] &
M^t_{a,b}\ar[d]^\sigma \ \\
\S^1 \ar[r] & Q_t \ar[r]  &
N_t \  \\
\Sph^3 \ar[r] & S_{a,b} \ar[r]^\pi  &
\CP^2  }
\end{split}
\end{equation*}

\smallskip

In particular, recall that we can regard the $2$-sphere bundle $N_t$ as the projectivization of
a rank $2$ complex vector bundle
$P_t \times_{\U(2)}\C^2$ with $c_1(P_t) = x\,,\,c_2(P_t)=t\,x^2$ and that
$N_t=P_t/\T^2$ as well. We defined $M^t_{a,b}$ as the circle bundle over $N_t$ with
Euler class $e=ax+(a+b)y$.

\newpage

\bigskip

\begin{center}
{\it A different description of $M^t_{a,b}$.}
\end{center}

\bigskip

\begin{prop}\label{same}
The circle bundle $\S^1\to M^t_{a,\, b }\to N_t$ can be equivalently
described as the circle bundle $\T^2/\S^1_{a,\, b }\to
P_t/\S^1_{a,\, b }\to P_t/\T^2$, where
$\S^1_{a,b}=\diag(z^{a},z^{b})\subset\T^2\subset \U(2)$.
\end{prop}
\begin{proof}
We first claim that, using the  basis $x,y\in H^2(P_t/\T^2)$ from Section 4,
the first Chern class $c_1(P_t/\S^1_{a,\, b })={r}x+{s}y$ for some
functions ${r},{s}$ linear in $a,\, b $. To see this, observe that
$H^*(P_t)\cong H^*(\Sph^3\times\Sph^5)$ since in the spectral
sequence of the principal bundle $\U(2)\to P_t\to\CP^2$, the
differential $d_2\colon H^1(\U(2))\to H^2(\CP^2)$ takes a
generator to $c_1(P_t)= x$. This holds for any $\U(2)$ principal bundle,
as can be seen by observing that this is true in the universal bundle and
 hence via pullback for any $\U(2)$-bundle. Now consider the commutative diagram of fibrations:
\begin{equation*}
\begin{split}
\xymatrix{
\T^2 \ar[r] \ar[d]^f& P_t \ar[r] \ar[d] &
P_t/\T^2 \ar[d]^{id}\ar[r]^{g_1} &  B_{\T^2}\ar[d]^{B_f} \\
\S^1\simeq \T^2/\S^1_{a,\, b } \ar[r] & P_t/\S^1_{a,\, b }
\ar[r]^{h} & P_t/\T^2 \ar[r]^{g_2} & B_{S^1} }
\end{split}
\end{equation*}
where $g_1$ and $g_2$ are the classifying maps of the respective
$\T^2$ bundle and $\S^1$ bundle. If we choose bases $\lambda\in
H^1(\S^1)$ and $\mu,\nu\in H^1(\T^2)$, the projection $f$ induces a
map $f^*\colon H^1(\S^1)\to H^1(\T^2)$ with
$f^*(\lambda)=\tilde{r}\mu+\tilde{s}\nu$ for some functions
$\tilde{r},\tilde{s}$ linear in $a,\, b $. Thus via transgression $
H^1(\S^1)\cong H^2(B_{\S^1})$ and $H^1(\T^2)\cong H^2(B_{\T^2})
   $   it follows that $B_f^*(\bar{\lambda})= \tilde{r}\bar{\mu}+\tilde{s}\bar{\nu}$.
 From the spectral sequence of $P_t
\to P_t/\T^2 \to B_{\T^2}$ it follows that $g_1^*$ is an isomorphism
in $H^2$. Since $g_2^*$ takes the canonical generator in
$H^2(B_{\S^1}) \cong \Z$ to $c_1( P_t/\S^1_{a,\, b })$, it follows that
$c_1(P_t/\S^1_{a,\, b
})=\tilde{r}g_1^*(\bar{\mu})+\tilde{s}g_1^*(\bar{\nu})$. Via a basis
change the claim follows.

\smallskip

We now show that $c_1(P_t/\S^1_{0,1})=y$ and $c_1(P_t/\S^1_{1,-1})=x$
 and thus $c_1(P_t/\S^1_{a,\; b})=ax+(a+b)y$
which implies that $M^t_{a,\, b }= P_t/\S^1_{a,\; b}$. To evaluate
the two Euler classes, we need to specify the orientation of the
circle bundle $\T^2/\S^1_{a,\, b } \to P_t/\S^1_{a,\, b } \to
P_t/\T^2$. If the Lie algebra of $\T^2$ is endowed with its natural
orientation, the Lie algebras of $\S^1_{a,\, b }$ and $\S^1_{-b,a}$
form an oriented basis. Thus the action of
$\diag(e^{-b\theta},e^{a\theta})\subset\U(2)$ on $P_t$ induces a
(possibly ineffective)   circle action on $P_t/\S^1_{a,\; b}$, which
is the orientation we will use in the following.

To see that $c_1(P_t/\S^1_{0,1})=y$, recall that $y=c_1(S^*)$ where
$S$ is the canonical line bundle over $P(E)=P_t\times_{\U(2)}\CP^1$.
 Consider the circle bundle
 $$\S^1\simeq \T^2/\diag(1,z)\to\U(2)/\diag(1,z)
 \simeq\Sph^3\to \U(2)/\T^2\simeq \CP^1.$$
  We identify $\U(2)/\diag(1,z)$ with $\Sph^3$
  by sending $A\in\U(2)$ to its first column vector. Thus the
  action of $\diag(e^{i\theta},1)$ on $\U(2)/\diag(1,z)$ from the right
  is multiplication by $e^{i\theta}$ in both coordinates and hence
  the Hopf action, which shows that this is
the canonical line bundle over $\CP^1$. It  follows that
$$\T^2/\diag(1,z)\to P_t\times_{\U(2)}\U(2)/\diag(1,z)\to P_t
\times_{\U(2)}\U(2)/\T^2\simeq P(E)$$
 is the canonical line bundle over
$P(E)$. The projection onto the first coordinate induces an
isomorphism $P_t\times_{\U(2)}\U(2)/\diag(1,z)\simeq
P_t/\diag(1,z)$, with circle action given by right multiplication
with $\diag(e^{i\theta},1)$. But this is the opposite orientation to
the oriented circle action on $P_t/\S^1_{0,1}\to P_t/\T^2$. Thus
this bundle is dual to the canonical line bundle, which proves our
claim.

\smallskip

To see that $c_1(P_t/\S^1_{1,-1})=x=\pi^*(x)$, consider the diagram of circle
fibrations:
\begin{equation*}
\begin{split}
\xymatrix{
 \U(2)/\SU(2) \ar[r] \ar[d]&
{(P_t)}\times_{\SU(2)} \CP^1 \ar[r] \ar[d]^{\pi_1} &
{(P_t)\times}_{\U(2)} \CP^1\simeq N_t \ar[d]^{\pi_1=\pi} \\
\U(2)/\SU(2) \ar[r] & P_t/\SU(2)\ar[r]
 & P_t/\U(2)\simeq \CP^2 }
\end{split}
\end{equation*}
The fibers of these bundles are oriented via the isomorphism induced
by the homomorphism $\det:\U(2)\to\S^1$ and hence right
multiplication by $\diag(e^{i\theta},1)\subset\U(2)$ induces a
circle action with the correct orientation. The lower circle bundle
is the Hopf bundle since by \eqref{cherncircle} it has Euler class $c_1(P_t)=x$.   Thus $c_1=x$ for the upper
circle bundle as well. But the total space is identified with
$(P_t)\times_{\SU(2)} \CP^1 \simeq (P_t)\times_{\SU(2)}
\SU(2)/\diag(z,\bar{z})\simeq P_t/\diag(z,\bar{z})\simeq
P_t/\S^1_{1,-1}$. The circle action by $\U(2)/\SU(2)$ is the right
action by $\diag(e^{i\theta},1)$ on $P_t/\S^1_{1,-1}$, whereas the
natural circle action is given by right multiplication with
$\diag(e^{i\psi},e^{i\psi})$. To see that both circle actions agree,
observe that
$\diag(e^{i\theta},1)=\diag(e^{i\psi},e^{i\psi})\cdot\diag(z,\bar{z})$
for $z=e^{i\psi}\; , \; \theta=2\psi$ and that the action by
$\diag(e^{i\psi},e^{i\psi})$ is $\Z_2$ ineffective.
\end{proof}

\begin{rem*}
There is another natural basis $\bar x,\bar y$ of $H^2(N_t,\Z)$
given by transgression in the fiber bundle $\T^2\to P_t\to N_t$ of
the natural basis in $H^1(T^2,\Z)$ corresponding to the splitting
$T^2=\diag (e^{i\theta},e^{i\psi})\subset\U(2)$. The Euler class of
the circle bundle $\T^2/\S^1_{a,\, b }\to P_t/\S^1_{a,\, b }\to
P_t/\T^2$ is then given by $-b\bar x +a \bar y$. Thus \pref{same}
implies that $\bar x=-y,\ \bar y=x+y$ and hence $e(M^t_{a,b})=a\bar
x+b\bar y$.
\end{rem*}

In \cite{GZ2} it was shown that $\U(2)$ principal bundles over
$\CP^2$ with $w_2\ne 0$ admit a metric with non-negative sectional
curvature invariant under the action of $\U(2)$. Hence, as a
consequence of \pref{same} and O'Neil's formula we obtain:

\begin{cor}\label{K >= 0}
The manifolds $M^t_{a,\, b }$ admit a metric with non-negative
sectional curvature for any integers $a,\, b\, ,t$.
\end{cor}

Since $\S^1_{a,b}\subset\U(2)$ is conjugate to $\S^1_{b,a}\subset\U(2)$, we have

\begin{cor}\label{symm}  The manifolds $M^{t}_{a,\, b }$ and $M^{t}_{b,\, a }$
are orientation preserving diffeomorphic.
\end{cor}

\bigskip

\begin{center}
{\it Relationship with previously defined manifolds.}
\end{center}

\bigskip

\begin{prop}\label{t=pm 1}
For $t=\pm 1$ one has the following identifications. $N_{1}$ is the
homogeneous flag manifold $\SU(3)/\T^2$ and $M^{1}_{a,\, b }$ the
Aloff-Wallach space $W_{a,\, b }$.
 $N_{-1}$ is the
inhomogeneous flag manifold $\SU(3)/\!/\T^2$ and $M^{-1}_{a,\, b }$
is the Eschenburg space $F_{a,\, b }$.
\end{prop}
\begin{proof}  If we start with the homogeneous $\U(2)$ principal bundle
$\U(2)\to\SU(3)\to\CP^2$, the associated 2-sphere bundle is
$\SU(3)/\T^2\to\CP^2$ and $\SU(3)/Z\to\CP^2$ is  its $\SO(3)$
principal bundle. But $\SU(3)/Z\simeq\SU(3)/\diag(z,z,\bar
z^2)\simeq W_{1,1}$. Since $W_{1,1}$ is simply connected and $
H^4(W_{1,1},\Z) \cong \Z_3$, \eqref{SO3} and \eqref{SO3p1} imply that the
$\SO(3)$ bundle has  $p_1=-3$ and $w_2\ne 0$. Thus $\SU(3)/Z=Q_1 $
and hence $\SU(3)/\T^2=N_1 $.

\smallskip

Similarly, there is a second free biquotient action of $\U(2)$ on
$\SU(3)$, see \cite{E2}, given by
$$B\star A=\diag(1,1,\det B^2)\ A\; \diag(B,\det B)^{-1} \text{
where }  B\in\U(2) \; , \;A\in\SU(3).$$ By dividing by $\SU(2)$
first, one easily sees that $\SU(3)/\!/ \U(2)\simeq \CP^2$, i.e. we
obtain a $\U(2)$ principal bundle over $\CP^2$. Now
 $\SU(3)/\!/Z=\diag(1,1,z^4)\backslash
\SU(3)/\diag(z,z,z^2)^{-1}= F_{1,1}$.  Since $F_{1,1}$ is simply connected and $ H^4(F_{1,1},\Z) \cong \Z_5$, \eqref{SO3}
and \eqref{SO3p1} imply that the $\SO(3)$ bundle has  $p_1=5$. Thus
$\SU(3)/Z=Q_{-1} $ and hence $\SU(3)/\T^2=N_{-1} $.

Thus in both cases $P_{\pm}=\SU(3)$, but with different actions by
$\U(2)$. Dividing by $\S^1_{a,b}$, the last claim follows from
\pref{same}.
\end{proof}

\begin{rem*}
The diffeomorphism classification of $M^1_{a,b}$ was carried out in
\cite{KS2}. Their choice of parameters $a,b$ is the same as ours,
but their  orientation is opposite. Notice though that in the choice
of $m,n$, one needs to change the sign of $n$, i.e. $am+bn=1$ in the
formulas in \pref{KSCircle}.

The diffeomorphism classification of $M^{-1}_{a,b}$ was carried out
in \cite{AMP1}. The correspondence  of our and their parameters is
$a=l,\ b=m,\ m=a,\ n=b$, and their  orientation is opposite. Notice
though that the invariants of their Examples 7-9 in the Table on
page 47 are incorrect.
\end{rem*}

The case of $t=0$ is also special:
\begin{prop}\label{t=0}
For $t=0$ we have $N_0=\CP^3 \, \# \, \overline{\CP^3}$ and
$P_0=\Sph^5 \times \Sph^3$. Furthermore,
$M^0_{a,b}=\Sph^5 \times \Sph^3/\S^1$ with circle action $(p,q)\to
(z^{a+b}p \ ,\diag(z^a, {z}^b)\, q )$ where $p\in\C^3\, , q\in\C^2$.
\end{prop}
\begin{proof}
The bundle  $P_0$ has $c_1=x$ and $c_2=0$ and its structure group
thus reduces to $\U(1)$. This reduced circle bundle must  be the
Hopf bundle since its Euler class is $x$. Hence
$N_0=\Sph^5\times_{\S^1}\Sph^2  $ which is well known to be
diffeomorphic to $\CP^3\#\overline{\CP^3}$.

Furthermore, $P_0=\Sph^5\times_{\U(1)}\U(2)$ where $\U(1)$ acts on
$\Sph^5$ as the Hopf action and on $\U(2)$ by left multiplication
with $\diag(z,1)$. One easily shows that the map
$\Sph^5\times\Sph^3=\Sph^5\times\SU(2)\to
\Sph^5\times_{\U(1)}\U(2)=P_0 $ given by $(p,A)\to [(p,A)]$ is a
diffeomorphism. This easily implies that the action of $\U(2)$ on
$P_0$,  translated to $\Sph^5\times\Sph^3$, is given by
  $B\star (p,A)\to (\det B\, p\, , \diag(\det\bar{B},1)\, A\, B)$
where $p\in\Sph^5\subset \C^3\, , A\in\SU(2)\simeq\Sph^3$ and
$B\in\U(2)$. If we identify  $\SU(2)$ with $\Sph^3$ via its first
column vector, we see that the circle action of $B=\diag(z^a,z^b)\in
\S^1_{a,b}\subset\U(2)$ on $P_0=\Sph^5\times\Sph^3$  sends
$(p,(u,v))$ to $(z^{a+b}p,(\bar{z}^b \, u,z^a\, v))$, where
$(u,v)\in\Sph^3\subset \C^2$. But the circle action $(u,v)\to
(\bar{z}^b \, u,z^a\, v)$ on $\Sph^3$ is equivalent to $(u,v)\to
({z}^a \, u,z^b\, v)$ via conjugation on the first coordinate and a
coordinate interchange. Since $M^0_{a,b}=P_0/\S^1_{a,\, b }$,  we
obtain the last claim.
\end{proof}

\begin{rem*}
Thus $M^0_{a,b}$ are special examples of the 5-parameter family of
manifolds mentioned at the beginning of Section 6. Together with
\cref{MandL} this implies that $M^0_{1,1}\simeq L_{-1,2}\simeq \bar
M_{-1,2}$.
\end{rem*}

 We can regard the principal bundle $Q_t$ as $P_t/Z$ where $Z=\S^1_{1,1}$  is the
center of $\U(2)$ and thus \pref{same} implies
\begin{cor}
The $\SO(3)$ principal bundle $Q_t$ with $p_1(Q_t)=1-4t\; ,
w_2(Q_t)\ne 0$ is equal to $M^t_{1,\; 1}$. Furthermore,
$Q_1=W_{1,1}$ and $Q_{-1}=F_{1,1}$.
\end{cor}
\smallskip

\bigskip

\begin{center}
{\it Natural diffeomorphisms between sphere and circle bundles.}
\end{center}

\bigskip

From the inclusions $\diag(z^{p},z^{q})\subset
 \U(2) $ we obtain
the  fibration:
$$ \Sph^3/\Z_{p+q}\simeq \U(2)/\diag(z^{p},z^{q})  \to P_t/\diag(z^{p},z^{q}) \simeq M^t_{p,q} \to P_t/\U(2)\simeq
\CP^2,$$
 where the fiber is
$\U(2)/\diag(z^{p},z^{q})=\SU(2)/\diag(z^{p},z^{q})$ with
$z^{p+q}=1$.  Hence the fiber is a  lens space $\Sph^3/\Z_{p+q}$,
and, in the case of $p+q=\pm 1$, we obtain a bundle with fiber
$\Sph^3$. We can assume that $p+q= 1$, since replacing $z$ by
$\bar{z}$ changes the sign of both $p$ and $q$. Thus we obtain sphere bundles:
\begin{equation}\label{bundle}
\Sph^3\to M^t_{p,1-p} \to
\CP^2.
\end{equation}
We now identify which sphere bundle this is.
\begin{prop}\label{diff}
 $M^t_{p,1-p}$ is orientation preserving diffeomorphic to $S_{-t,p(p-1)}$.
\end{prop}
\begin{proof}
 We compute the characteristic classes of the $\SO(4)$
principal bundle $P^*$ over $\CP^2$ associated to the sphere bundle
\eqref{bundle}. For this we  identify the induced $\SO(3)$ principal
bundles $P_\pm^*=P^*/\S^3_\pm$ as discussed  in Section 1. Notice
though that the role of $P$ and $P^*$ are interchanged here.

 Let $\rho_p\colon
\U(2)\to \U(2)$ be the homomorphism $\rho_p(A)=(\det A)^{-p}A$ which
we can also regard as  a representation of $\U(2)$ on $\C^2$. The
induced action of $\U(2)$ on the unit sphere $\Sph^3(1)\subset
\R^4\cong \C^2 $ is transitive with isotropy group at $(1,0)$ given
by $\S^1_{p}:=\diag(z^p,z^{1-p})$, i.e., $\Sph^3(1)=\U(2)/\S^1_p $.
Consider now the associated vector bundle $E_p=P_t\times_{\U(2)}
\C^2$ where $\U(2)$ acts on $\C^2$ via $\rho_p$ and
$E_p^\R=P\times_{\U(2)} \R^4$ the underlying real bundle. For the
sphere bundle of $E_p^{\R}$ we have $S(E_p^{\R})=P\times_{\U(2)}
\Sph^3=P\times_{\U(2)} \U(2)/\S^1_p=P/\S^1_p$ and thus, using \pref{same},
$S(E_p^{\R})=M^t_{p,1-p}$. We denote by $P^*$ the $\SO(4)$ principal bundle associated to
$E_p^\R$, i.e.  $P^*=P\times_{\U(2)} \SO(4)$.

\smallskip

In order to view $P^*$ in a different way, consider  the following commutative diagram of
 homomorphisms:

\begin{equation*}
\begin{split}
\xymatrix{ \S^1\times\S^3 \ar[d]_{\pi_1}\ar[r]^{\tilde{\rho}_p}&
\S^1\times\S^3 \ar[d]_{\pi_1} \ar[r]^{\tilde{\gs}}&
\S^3\times\S^3 \ar[d]^{\pi_2}  \\
\U(2) \ar[r]^{\rho_p} &  \U(2)  \ar[r]^{\gs} & \SO(4)  }
\end{split}
\end{equation*}
 where $\pi_i$ are the
 two fold covers and $\gs,\  \tilde{\gs}$  the embeddings discussed in Section 1. To make
 these diagrams commutative, one needs
$\tilde{\rho}_p(z,q)=(z^{-2p+1},q)\in \S^1\times\S^3$.

\smallskip

Thus $P^* = P\times_{\U(2)}
\SO(4)=P\times_{(\S^1\times\S^3)/\Gamma}(\S^3\times\S^3)/\Gamma$ and
$(z,q)\in\S^1\times\S^3$ acts on $(q_1,q_2)\in\S^3\times\S^3$ as
$(z^{-2p+1}q_1,qq_2)$.  According to \eqref{chernpm}
$$P^*_-= P/Z=
Q_t$$ and thus  $p_1(P_{-}^*) =p_1(Q_t)=
(1-4t)$. On the other hand, by \eqref{chernpm}
$$P_+^*= (P/\SU(2))\times_{\SO(2)}\SO(3) =
(P/{\SU(2)})\times_{(\S^1/\Gamma)\times\{e\}}\left[(\S^3/\Gamma)\times\{e\}\right].
$$
 but $\S^1$ acts on $\S^3$ via $z\star q = z^{-2p+1}q$.
 The structure group of this bundle reduces to $\S^1$ with $\S^1$ principal bundle
 \begin{equation}\label{reduced}
 \S^1\to (P/{\SU(2)})\times_{(\S^1/\Gamma)\times\{e\}}\left[(\S^1/\Gamma)\times\{e\}\right]=
 (P/{\SU(2)})/\Z_{2p-1} \to P/\U(2)
 \end{equation}
 since the action of $\S^1$ on $\S^1$ has isotropy $\Z_{2p-1}$.
Recall that by \eqref{cherncircle}  the circle bundle $P/\SU(2)\to P/\U(2)$ has Euler class
$c_1(P)=x$. The Gysin sequence then implies that the circle bundle \eqref{reduced} has Euler
class $e=\pm (2p-1)x$ and hence
$p_1(P_+^*)=e^2=(1-2p)^2$. By definition, $p_1(P_-^*)=4a+1$ and
$p_1(P_+^*)=4b+1$,  and hence $a=-t$ and $b=p(p-1)$, i.e.,
 $M^t_{p,1-p} = S_{-t,p(p-1)}\,$.

In order to see that the diffeomorphism is orientation preserving, we use \cref{linkingS3} and
\cref{linkingS1} to see that the linking forms are the same.
\end{proof}

In particular:
\begin{cor}\label{sameAW} For each integer $p$, there exist natural
diffeomorphisms $W_{p,1-p} \simeq S_{-1,p(p-1)}\,$ and $F_{p,1-p} \simeq
S_{1,p(p-1)}$.
\end{cor}

\begin{rems*}
(a) Using the normalization of the Aloff-Wallach spaces described in
Section 1, we see that $W_{p,1}\simeq S_{-1,p(p+1)}$ with $p\ge 0$
are the Aloff-Wallach spaces which are naturally  $\Sph^3$ bundles
over $\CP^2$. They have positive curvature, unless $p=0$. Using the
diffeomorphism classification in \cref{hdclass}, this gives rise to
infinitely many diffeomorphisms between sphere bundles and
$W_{p,1}$.

(b) The
tangent bundle of $\CP^2$ is given by $S_{-1,2}\simeq W_{2,-1}\simeq
W_{1,1}$ since it has $p_1=e=3$. Thus $W_{1,1}$ is both an $\S^3$
bundle over $\CP^2$,  as well as an $\SO(3)$ principal bundle over $\CP^2$.

(c) The Eschenburg spaces $F_{p,1-p}$ do not have positive curvature
in the Eschenburg metric since $p(1-p)\le 0$.
\end{rems*}

\bigskip

\section{Spin $\Sph^3$ and $\Sph^1$ bundles over $\CP^2$}

\bigskip

In this section we discuss bundles which are spin. One class of such
bundles has a total space which is spin and can hence be
diffeomorphic to an Eschenburg space. Some of the other classes
admit Einstein metrics and we will compare them as well.

\smallskip

There is a specific class of manifolds of type $\bar{E}_r$ which has been considered previously in
 \cite{WZ}, where it was shown to admit Einstein metrics. These manifolds are total spaces
 of circle bundles $\S^1\to L_{a,b}\to
 \CP^2\times\CP^1$
with Euler class $e=ax+by$ and $\gcd(a,b)=1$, where $x$ and $y$ are
the natural generators in the first or second factor. They can also
be considered as the base space of a circle bundle: $\S^1\to
\Sph^5\times\Sph^3\to L_{a,b}$ where $\S^1$ acts on
$\Sph^5\times\Sph^3\subset \C^3\oplus \C^2$ as $(u,v)\to (z^{-b} u,
z^{a}v)$. Thus $L_{1,0}= \Sph^5 \times\CP^1  $ and $L_{0,1}=
\CP^2\times\Sph^3$, and by conjugating in each component, one sees
that $L_{a,b}=L_{\pm a,\pm b}$. Furthermore,  $L_{a,b}$, $a\ne 0$,
has cohomology type $E_{a^2}$ if $b$ is even and cohomology type
$\bar{E}_{a^2}$ if $b$ is odd.   Projecting to the first component,
one obtains a lens space bundle $\Sph^3/\Z_{|b|}\to L_{a,b} \to
\CP^2$. Thus $L_{a,1}$ is naturally an $\Sph^3$ bundle over $\CP^2$.
See \cite{WZ} for details.
\smallskip

These manifolds were later also discussed in \cite{KS1}, where they
computed the Kreck-Stolz invariants and exhibited certain
diffeomorphism among them, which gave rise to counter examples to a
conjecture by W.Y.Hsiang.

\smallskip

In \cite{Kr1} the manifolds $L_{a,b}$ were generalized to a $5$
parameter family of manifolds by dividing $\Sph^5\times\Sph^3$ by
the $\S^1$ action $((u_1,u_2,u_3),(v_1,v_2))\to
 ((z^{a_1} u_2,z^{a_2} u_3,z^{a_3}), (z^{b_1} v_1,  z^{b_2} v_2))$, and in  \cite{Es},  their Kreck-Stolz invariants were
 computed if all $a_i$ are equal. These manifolds have cohomology type
  ${E}_r$ or  $\bar{E}_r$, depending on whether $\sum a_i+\sum b_i$ is even or odd, and  $r=b_1b_2$.

\bigskip

\subsection{3 sphere bundles which are spin}

Here we consider the bundles
$$\Sph^3\to \bar{S}_{a,\, b } \to \CP^2 \text{ with } p_1(\bar{S}_{a,\, b })=2a+2b \ , \  e(\bar{S}_{a,\, b })
=a-b \text{ and } w_2=0$$
Recall that, according to \eqref{SO4spin}, these are the allowed values in the spin case.
As in Section 3, one easily proves:

\begin{prop}\label{topspin}  If $r=|a-b|\ge 1$ one has the
following:
\begin{itemize}
\item[(a)] The manifolds $\bar{S}_{a,\, b }$  have cohomology type $\bar{E}_r$
 and $p_1(T\bar{S}_{a,\, b })\equiv 2a+2b+3 \text{ mod } r$.
 Furthermore,  the linking form  is given by $lk(\bar{S}_{a,\, b
 })=\frac{1}{a-b} $.
 \item[(b)] $\bar{S}_{a,\, b }$ has non-negative curvature if $a$ and $b$ are both even.
\end{itemize}
\end{prop}
\no Part (b) follows from \cite{GZ2}. It is not known whether
$\bar{S}_{a,\, b }$ admits a metric with non-negative curvature if
$a$ and $b$ are not both even, although they do if $a$ is even and
$b=(2r+1)^2$.

\smallskip

For the Kreck-Stolz invariants one obtains:

\begin{prop}\label{KSspherenonspin}
The Kreck-Stolz invariants for $\bar{S}_{a,\, b}$ with $a\ne b$ are
given by:
$$
\begin{aligned}
s_1(\bar{S}_{a,\, b }) &  \equiv \frac{1}{2^7 \cdot 7 \cdot (a-b)}\,
( 2a+2b+3 )^2- \frac{1}{2^7 \cdot 3 \cdot (a-b)}\,
( 4a+4b+5 ) - \frac{\signs(a-b)}{2^5\cdot 7}  \mod 1 \\
s_2(\bar{S}_{a,\, b}) &  \equiv  \frac{-1}{2^2 \cdot 3 \cdot (a-b)}\,(a+b-1) \mod 1 \\
s_3(\bar{S}_{a,\, b }) &  \equiv \frac{-1}{2^2 \cdot (a-b)}\,(a+b-5) \mod 1.
\end{aligned}
$$
\end{prop}

This easily implies the following homeomorphism and diffeomorphism
classification:

\begin{cor}\label{hdclass} The manifolds $\bar S_{a,b}$ and $\bar S_{a',b'}$ with $r=a-b=a'-b'>0$ are
\begin{itemize}
\item[(a)]
 orientation preserving  homeomorphic  if and only if  $a\equiv
 a' \mod 6\, r$.
 \item[(b)] orientation preserving diffeomorphic  if and
only if $$a\equiv a' \mod 6\, r  \ \text{ and } \ (a-a')\,[3 (a
 + a') - 3r +1]\equiv 0 \mod 2^3\cdot 3\cdot 7\cdot r.$$
\item[(a')]  orientation
reversing  homeomorphic  if and only if  $r=1$ and $a+a'\equiv 2
\mod 6$ or $r=2$ and $a+a'\equiv 3 \mod 12$.
 \item[(b')]  orientation  reversing diffeomorphic  if and
only if
$$
a+a'\equiv 2 \mod 6 \  \text{ and } \   a(3a-2)\equiv -a'(3a'-2) +2 \mod 2^3\cdot 3\cdot 7\,  \text{ for  } r=1
$$
or
$$
a+a'\equiv 3 \mod 12  \  \text{ and } \  a(3a-5)\equiv -a'(3a'-5) \mod 2^4\cdot 3\cdot 7\,   \text{ for  } r=2 \,.
$$
\end{itemize}
\end{cor}

\bigskip

\subsection{Circle bundles over 2 sphere bundles which are spin}

In this section we discuss the manifolds $\bar M^t_{a,\, b }$ arising from $\Sph^2$ bundles over $\CP^2$ which are spin.  According to \eqref{SO3}, the
corresponding $\SO(3)$ principal bundle  satisfies $p_1\equiv 0 \mod
4$ and we define:
$$
\Sph^2\to \bar N_t\overset{\pi}{\longrightarrow} \CP^2
 \quad \text{ with } \quad p_1(\bar N_t)=4\,t\,x^2\quad \text{ and } \quad w_2 = 0
$$
for some integer $t$ and we denote the corresponding $\SO(3)$ principal bundle by $\bar Q_t$.
 If we define the $\U(2)$ principal
 bundle $$\U(2)\to\bar P_t\to \CP^2 \quad \text{ with }\quad c_1(\bar P_t)=0  \quad \text{ and }
  \quad c_2(\bar P_t)=-tx^2$$
 we again have that $$\bar{N}_t\simeq P(\bar E)\simeq \bar P_t/\T^2\quad \text{ and }\quad \bar P_t/Z=\bar Q_t$$ where $\bar E=\bar P_t\times_{\U(2)}\C^2$.
 For the cohomology ring of  $\bar N_t$ we thus obtain
\begin{equation}\label{cohbasespin}
\begin{aligned}
H^2(\bar N_t) &\cong \Z \oplus \Z \,\text{with generators}\, x,y\,;\\
H^4(\bar N_t) &\cong  \Z \oplus \Z \,\text{with generators}\,
x^2,xy\,\text{and relationship}\, y^2 = t\,x^2\,;\\
H^6(\bar N_t) &\cong  \Z \,\text{with generator}\, x^2y\,\text{and}\,
x^3=0,y^2\,x = 0,y^3=t\,x^2y\,.\\
\end{aligned}
\end{equation}
where $y=c_1(S^*)$ and $S^*$ is the dual of the
tautological complex line bundle $S$ over $P(\bar E)$.

\smallskip

\bigskip

We now define the circle bundles over $\bar N_t$ via
$$\Sph^1\to  \bar M^t_{a, b}
 \overset{\sigma}{\longrightarrow}  \bar N_t
 \quad \text{ with }\quad  e(\bar M^t_{a, b} )=ax +by
$$
and $(a,\, b ) =1$. One of the differences with $N_t$ is that $\bar N_t$ is  not spin and hence
$\bar M^t_{a, b} $ may or may not be spin. One easily sees

\begin{prop}\label{circletopspin}
If $r=|a^2-t\,b^2|$ with $r\ne 0$, the manifolds $\bar M^t_{a,\, b }$ have cohomology type $\bar E_r$
if $b$ is odd, and  cohomology type $ E_r$ if $b$ is even. Furthermore, the
 first Pontryagin class  is $p_1(T\bar M^t_{a,\, b })\equiv
(3+4\,t)\,b^2 \text{ mod } r$.
\end{prop}

\begin{rem*} The homotopy invariant $\pi_4(\bar
M^{2t}_{a,b})=\pi_4(\bar N_t)$ again does  not depend on $a,b$. We suspect
that, as in the case of non-spin bundles, $\pi_4(\bar M^t_{a,b})=0$
if $t$ odd, and $\pi_4(\bar M^{t}_{a,b})=\Z_2$ if $t$ even. For
example, it follows from \cref{MandL} that $\pi_4(\bar
M^{0}_{a,b})=\Z_2$.
\end{rem*}

In order to describe the circle bundles in a different way, as in
\pref{same},  we consider $\bar N_t$ as the projectivization of the
vector bundle $\bar E=\bar P_t^*\times_{\U(2)}\C^2$ where $c_1( \bar
P_t^*)=2x$ and $c_2( \bar P_t^*)=-(t-1) x^2$. Since
$c_1^2-4c_2=4tx^2$ and since $w_2\equiv c_1\mod 2=0$,
 the
projectivization of the vector bundle associated to $\bar P^*_t$ is again $\bar N_t$ and
$\bar P^*_t/\T^2=\bar N_t$
as well. But notice that $\pi_1(\bar P_t^*)\cong \Z_2$ since in the spectral sequence for the
 $\U(2)$ principal bundle
 $d_2$ takes a generator in $H^1(U(2),\Z)\cong \Z$ to $c_1=2x$. Similarly, since $w_2(\bar Q_t)=0$,
 \eqref{SO3p1} implies  that $\pi_1(\bar Q_t)\cong \Z_2$.
 We let $\bar P'_t$ and $\bar Q'_t$ be the universal covers of  $\bar P^*_t$ and $\bar Q_t$.
We thus obtain the spin bundles
$$
\S^1\times\S^3\to\bar P_t'\to \CP^2\quad , \quad \S^3\to \bar
Q_t'\to\CP^2 \qquad \text{ with } \qquad \bar P_t'/\T^2=\bar N_t = \bar Q_t'/\SO(2).
$$

We can now formulate the analogue to \pref{same}:

\begin{prop}\label{samespin}
The circle bundle $\S^1\to \bar M^t_{a,\, b }\to \bar N_t$ can be equivalently
described as the circle bundle $\T^2/\S^1_{-b,\, a }\to
\bar P_t'/\S^1_{-b,\, a }\to \bar P_t'/\T^2$, where
$\S^1_{a,b}=\diag(z^{a},z^{b})\subset\S^1\times\S^1\subset \S^1\times\S^3$.
\end{prop}
\begin{proof}
The proof is similar to the proof of \pref{same} and we indicate the changes that need to be made as
one goes from $\U(2)$ principal bundles to $\S^1\times\S^3$ principal bundles.

In the proof of \pref{same} we showed that if $P\to\CP^2$ is a
$\U(2)$ principal bundle,  the circle bundle $P/\S^1_{r,s}\to
P/\T^2$ has Euler class linear in $r$ and $s$ and
$e(P/\S^1_{0,1})=c_1(S^*)$ and $e(P/\S^1_{1,-1})=c_1(P)$. These
results did not depend on whether the bundle is spin or not. We now
apply this to the bundle $\bar P^*$ above.
  Notice though that if we describe $\bar N_t$ as $ \bar P^*_t/\T^2$,
the cohomology ring is expressed in terms of a different basis $x,\,
y'$ and now ${y'}^2=-2xy'+(t-1)x^2$, whereas in the $x,y$ basis we
have $y^2=tx^2$.
 On the other hand, the only elements  $z\in H^2(\bar N_t)$ with $z^2=t x^2$ are a multiple of $x$, or $z=\pm y$.
 But if $z=\alpha x+\beta y'$, we have $z^2=t x^2$  only if $\beta=0$,   or
 $\alpha=\beta=\pm 1$.  Since $y$ is not a multiple of $x$, this implies  that $z=\pm (x+y')$. Thus, depending on a sign
 $\e=\pm 1$,  $y=\e (x+y')$.
  As indicated above,
 $e(\bar P^*_t/\S^1_{0,1})=y'$
 and $e(\bar P^*_t/\S^1_{1,-1})=c_1(\bar P^*_t)=2x$. This implies that
  $e(\bar P^*_t/\S^1_{r,s})=2r x+(r+s)y'=(r-s)x+\e (r+s)y$ and hence
  $\bar P_t^*/\S^1_{r,\, s }=\bar M^t_{ ( r-s, \e\, (r+s) )}$.

  Recall that in the two fold cover $\S^1\times\S^3\to\U(2)$ the circle $\S^1_{a,b}$ is mapped to the circle
$\S^1_{a+b,a-b}$ and thus $\bar P'/\S^1_{a,b} \to \bar P^*/\S^1_{a+b,a-b}=\bar
M^t_{2b,\e 2a}$ is a two fold cover. But from the Gysin sequence it follows that
 $\pi_1(M^t_{2b,\e 2a})=\Z_2$ and thus $M^t_{b,\e a}$ is the only two fold cover and hence
 $\bar P'/\S^1_{a,b}= M^t_{b,\e a}$. In other words, $M^t_{a,b}$
 is orientation preserving diffeomorphic to $\bar P'/\S^1_{\e b,a}$.

 We now claim that we can make
 an arbitrary choice in the value of $\e$. Indeed, by conjugating with $(1,j)\in\S^1\times\S^3$, we see that the circles
$\S^1_{a,b}$ and
 $\S^1_{a,-b}$ are conjugate in $ \S^1\times\S^3$. But notice that this conjugation also reverses the
  orientation of the circle. Changing
 the sign of $a$ and $b$ changes the sign of the Euler class, and hence the
 orientation. Thus $\bar P'/\S^1_{a,b}$ and $\bar P'/\S^1_{-a,b}$ are
 orientation preserving diffeomorphic. We will make the choice of
 $\e =-1$ for symmetry reasons.
\end{proof}

\begin{rem*} In this description of $\bar M^t_{a,\, b} $ it is important that we
choose the bundle $\bar P_t' $ instead of either $\bar P_t$ or $\bar
P_t^*$. Indeed, following the same proof, one sees that
 $e(\bar P_t/\S^1_{1,-1})=0$
and hence $e(\bar P_t/\S^1_{r,s})=(r+s)y$, i.e. among the
circle bundles $\bar P_t/\S^1_{r,s}\to \bar N_t$ we obtain up to covers only
one bundle. In terms of the circle bundles $\bar P^*_t/\S^1_{r,s}\to
\bar N_t$, the proof also shows that $\bar P^*_t/\S^1_{r,s}= \bar
M^t_{r-s, - (r+s)} $, which only gives half of the manifolds $\bar
M^t_{a,\, b }$.
\end{rem*}

We can now discuss which of these manifolds carry a metric with non-negative curvature:

\begin{cor}
The manifolds $\bar M^t_{a,b}$ with $t$ even admit a metric with non-negative curvature.
\end{cor}
\begin{proof}  In \cite{GZ2} it was shown that any $\U(2)$ principal bundle with
$w_2=0$ and $c_1^2-4c_2$ divisible by $8$ admits a metric with
nonnegative curvature invariant under $\U(2)$. Since for $\bar
P_t^*$ we have  $c_1^2-4c_2=4t$,  it admits such a metric when $t$
is even. Thus the two fold cover $\bar P_t'$ admits a non-negatively
curved metric invariant under $\S^1\times\S^3$ and since by
\pref{samespin} we have $\bar P_t'/\S^1_{-b,\, a }=\bar M^t_{a,b}$,
and O'Neil's formula implies the claim.
\end{proof}

\smallskip
We now discuss various natural diffeomorphisms.

\begin{cor}
The manifold $\bar M^t_{a,b}$ is orientation preserving
diffeomorphic to $\bar M^t_{-a, b}$, and orientation reversing
diffeomorphic to $\bar M^t_{a, -b}$ and  $\bar M^t_{-a, -b}$. Thus
$\bar M^t_{\pm a, \pm b}$ are all diffeomorphic to each other.
\end{cor}
This follows from  \pref{samespin},  together with the observation
that $\bar P'/\S^1_{a,b}$ and $\bar P'/\S^1_{-a,b}$ are orientation
preserving diffeomorphic, see the end of the proof of
\pref{samespin}.

\smallskip

Next, recall that we have circle bundles $L_{a,b}\to
\CP^2\times\CP^1$ with Euler class $e=ax+by$.

\begin{cor}\label{MandL}
The manifold $\bar M^0_{a,b}$ is naturally diffeomorphic to $L_{a,b}$.
\end{cor}
\begin{proof}
The bundle $\bar N_t$ with $t=0$ is trivial since $p_1=0$ and
$w_2=0$. Thus $\bar N_0=\CP^2\times\CP^1$. In order to identify
$\bar P_0'$, we start with $R=\Sph^5\times\Sph^3$ and define a free
action by $\S^1\times\S^3$  on $R$ as follows.
$\S^1\subset\S^1\times\S^3$ acts as the Hopf action on the
 first factor  and  $\S^3\subset\S^1\times\S^3$ acts
via left multiplication on the second factor $\Sph^3$, regarded as
the Lie group $\Sp(1)$. Since $R/\S^1\times\S^3=\CP^2$ we thus have
a principal bundle $\S^1\times\S^3\to R\to\CP^2$ and we claim that
this principal bundle is $\bar P_0'$. To see this, we apply
\eqref{chernpm} to the $\U(2)$ principal bundle $R^*=R/\Z_2$ with
$\Z_2$ generated by $(-1,-1)\in\S^1\times\S^3$. Clearly,
$R^*/\SU(2)=\RP^5$ and hence the circle bundle \eqref{cherncircle}
has Euler class $2\, x$ which implies $c_1(R^*)=2\, x$. Furthermore,
$P_-=R^*/Z=R/\left[\S^1\times\{\pm 1\}\right]=\CP^2\times\SO(3)$.
Thus $0=p_1(P_-)=c_1^2-4 c_2$ which implies $c_2(R^*)=x^2$. Thus
$R^*=\bar P_0^*$ and hence $\bar P_0'=R=\Sph^5\times\Sph^3$. Now
\pref{samespin} implies our claim.
\end{proof}

There are again natural diffeomorphisms between the circle bundles and 3-sphere bundles as
in \eqref{diff}. From the inclusion $\S^1_{a,b}=\diag(z^{a},z^{b})\subset \S^1\times\S^3$ we obtain
lens space bundles:
$$ \Sph^3/\Z_{|b|}= \S^1\times\S^3/\diag(z^{-b},z^{a})  \to \bar P_t'/\diag(z^{-b},z^{a}) =
\bar M^t_{a,b} \to P_t/\S^1\times\S^3=
\CP^2,$$
which are 3-sphere bundles if $b=\pm 1$, and we can assume that $b=1$.

\begin{prop}\label{diffspin}
The manifold $\bar M^t_{k,1}$ is orientation preserving diffeomorphic to $\bar S_{t, k^2}$.
\end{prop}
\begin{proof}
The proof is similar to \pref{diff} and we indicate the changes one needs to make.
Let $\rho_k$ be the representation of $\S^1\times\S^3$ on $\R^4\simeq \QH$ given by $v\in \QH\to
z^k v\bar q,\ (z,q)\in \S^1\times\S^3$. We then have the associated vector bundle
$\bar P'_t\times_{ \S^1\times\S^3}\R^4$.  The action on the sphere $\Sph^3\subset\R^4$ is
transitive with isotropy $(z,z^{k})$
and hence the sphere bundle of this vector bundle  is diffeomorphic to
$\bar P'_t/\S^1_{1,k}=\bar M^t_{k,-1}$.

As in the proof of \pref{diff}, we let $\hat P_t=\bar P'_t\times_{
\S^1\times\S^3}(\S^3\times\S^3/\Gamma)$ be the $\SO(4)$ principal
bundle corresponding to this sphere bundle. Here $(z^k,q)\in
\S^1\times\S^3$ acts on $\S^3\times\S^3$ as left multiplication. The
remaining computations are similar:

$$
 \hat P_-=\left[\bar P'_t\times_{
\S^1\times\S^3}(\S^3\times\S^3)/\Gamma\right]/\S^3\times\{e\}= \bar P'_t\times_{
\S^1\times\S^3}\left[\{e\}\times(\S^3/\Gamma)\right]
$$

$$
\hspace{-150pt} = \hat P'_t/\left[\S^1\times\{\pm 1\}\right] =  \bar P^*_t/Z=\bar Q^*_t
$$
and thus $p_1(\hat P_-)=4t$.
Furthermore,
$$
 \hspace{-50pt}\hat P_+=\left[\bar P'_t\times_{
\S^1\times\S^3}(\S^3\times\S^3)/\Gamma\right]/\{e\}\times\S^3
=  \bar P'_t\times_{
\S^1\times\S^3}\left[(\S^3/\Gamma)\times\{e\}\right]
$$

$$
\qquad = (\bar P'/ \{e\}\times\S^3)_{  \S^1\times\{e\}} \left[(\S^3/\Gamma)\times\{e\}\right]
 = (\bar P^*/ \SU(2))_{  (\S^1/\Gamma)\times\{e\}} \left[(\S^3/\Gamma)\times\{e\}\right]
$$

\no This bundle reduces to the circle bundle

$$(\bar P^*/ \SU(2))_{  (\S^1/\Gamma)\times\{e\}} \left[(\S^1/\Gamma)\times\{e\}\right]=(\bar P^*/ \SU(2))/\Z_k
$$

Since the $\S^1$ bundle $\bar P^*/ \SU(2)\to \bar P^*/ \U(2)$ has Euler class $c_1=2x$, this
 reduced bundle
has Euler class $2kx$ and hence $p_1(\hat P_+)=4k^2$. Thus $a=t$ and $b=k^2$.
 To see that the diffeomorphism is orientation preserving we
use  \pref{topspin}
  and  \cref{linkingspin} to see that the linking form is the same.
\end{proof}

Finally we observe:
\begin{cor}
The principal $\S^3$ bundle over $\CP^2$ with Euler class $t x^2$ is
orientation preserving diffeomorphic to $\bar M^t_{0,1}\simeq \bar
S_{t,0}$.
\end{cor}
\begin{proof}
Recall from Section 1 that $\bar P_t/Z=\bar Q_t$ where $Z$ is the
center of $\U(2)$, and hence similarly $\bar
P'_t/\left[\S^1\times\{e\}\right]=\bar Q_t'$. Hence \pref{samespin}
implies that $\bar Q_t'=\bar M^t_{0,1}$. By assumption, we have
$p_1(\bar Q_t)=4t\, x^2$. This implies that $e(\bar Q_t')=t\, x^2$
since the homomorphism $\S^3\to\SO(3)$ induces multiplication by $4$
on $H^4(B_{\S^3},\Z)\simeq\Z\to H^4(B_{\SO(3)},\Z)\simeq\Z$ and the
generator in the first group is the Euler class and in the second
group the Pontryagin class.
\end{proof}

\bigskip

For the Kreck-Stolz invariants we have:

\begin{prop}\label{KSCirclespin}
If $ s = a^2  - b^2\,t\ne 0$, the Kreck-Stolz invariants for $\bar
M^t_{a,\, b }$  are given by:

\begin{center}
\vspace{5pt} $b$ even
\end{center}

$$
s_1(\bar M^t_{a,\, b })   \equiv\,   -\frac{1}{2^5 \cdot 7}\,\sign(W) +
                                  \frac{b}{2^7 \cdot 7}\,(6+8\,t+3\,a^2+ b^2\,t)
                   - \frac{b}{2^7 \cdot 7 \cdot s}\,(3 + 4\,t)^2  \mod 1
$$
$$
\begin{aligned}
 s_2(\bar M^t_{a,\, b })   \equiv\, & - \frac{1}{2^4\cdot 3}\, \{b\,(n^2+t\, m^2)
 - 2\,a\,n\,m\} \\
  & \quad\quad \quad\quad \quad   -       \frac{1}{2^4\cdot 3\cdot s}\,\{4\,n\,m\, \alpha
  - [3 + 4\,t - 2\,(n^2+t\, m^2)]\,\beta   \}
 \mod 1,\\
s_3(\bar M^t_{a,\, b })   \equiv\, &- \frac{1}{2^2\cdot 3}\,\{ b\,(n^2+t\, m^2) - 2\,a\,n\,m \}   \\
   & \quad\quad \quad\quad \quad  - \frac{1}{2^2\cdot 3\cdot s}\,\{16\,n\,m\,\alpha -
[3 + 4\,t - 8\,(n^2+t\, m^2)]\,\beta  \} \mod 1,
\end{aligned}
$$
\begin{center}
\vspace{20pt} $b$ odd
\end{center}

$$
\begin{aligned}
s_1(\bar M^t_{a,\, b })   \equiv\,  &- \frac{1}{2^5 \cdot 7}\,\sign(W) +
 \frac{b}{2^7 \cdot 7}\,(6+8\,t+3\,a^2+ b^2\,t)  - \frac{b}{2^7 \cdot 7 \cdot s}\,(3 + 4\,t)^2
                                 \\
       &
 \quad\quad \quad\quad \quad   - \frac{1}{2^6 \cdot 3} \,\{b\,(n^2 + t\,m^2) - 2\,a\,n\,m\}      \mod 1, \\
    &
 \quad\quad \quad\quad \quad +\frac{1}{2^7\cdot 3\cdot s}\,\{-2\,n\,m\,\alpha + (6 + 8\,t - n^2 - t\,m^2)\,\beta \} \\
 s_2(\bar M^t_{a,\, b })   \equiv\, &   -\frac{1}{2^3\cdot 3}\,  \{ b\,(n^2+t\, m^2) - 2\,a\,n\,m\}\\
 &   \quad\quad \quad\quad \quad  -       \frac{1}{2^3\cdot 3\cdot s}\,\{10\,n\,m\,\alpha  -
 [3 + 4\,t - 5(n^2+t\, m^2)]\,\beta  \}
 \mod 1, \\
s_3(\bar M^t_{a,\, b })   \equiv\, &- \frac{1}{2^3}\,\{b\,(n^2+t\, m^2) - 2\,a\,n\,m \} \\
&  \quad\quad \quad\quad \quad  - \frac{1}{2^3\cdot s}\,\{26\,n\,m\,\alpha - [3 + 4\,t -
13(n^2+t\, m^2) ]\,\beta  \} \mod 1,
\end{aligned}
$$
where   $$\alpha =a\, (n^2+t\, m^2) + 2\,t\,b\,n\,m \quad,\quad
\beta =  b\, (n^2+t\, m^2)    + 2\,a\,n\,m,$$ and $n,m \in \Z$ are
chosen such that  $a\,m + b\,n =
 1$.
 Furthermore, if $b$ is odd, we additionally require that $m$ is odd as well.
 Also,
$$ \sign(W)=\left\{
   \begin{array}{lll}
     &0, &\ \text{ if }\ s>0 \\
     &2, &\ \text{ if }\ s<0 \ \text{ and } \  b(t+1)>0 \\
     -&2, & \ \text{ if }\  s<0 \ \text{ and } \ b(t+1) <0.
   \end{array}
 \right.
$$
\end{prop}
\begin{proof} The proof is similar to \pref{KSCircle}, and we indicate the changes that are needed.
 A natural choice
 for a bounding manifold is the  disk bundle $\gs'\colon
 \bar W^8_{a,\, b } \longrightarrow \bar N_t$ of the rank 2 vector bundle $E^2$ associated
 to the circle bundle $\gs$.  One easily shows that
$$\begin{aligned}
p_1(T\bar W) &= (3 + 4\,t)\,x^2 + e^2 \ , \ w_2(T\bar W) \equiv (1+a)\,x
+ b\,y \mod 2 \,.
\end{aligned}$$
where  $e=a\,x +b\,y$.  Following the proof of \pref{circletop}
 there exists a generator $u\in H^2(\bar M^t_{a,\, b
};\Z)\cong\Z$ with $\sigma^*(x)=-bu$ and $\sigma^*(y)=au$ and
$w_2(T\bar M^t_{a,\, b }) = b\,u \mod 2$. Thus the manifolds $\bar
M^t_{a,\, b }$ and $\bar W_{a,b}$ are both spin if $b$ is even and
both non-spin if $b$ is odd.   Hence we can choose $c=0$ in
\eqref{KSspin} and \eqref{KSnonspin}.  Next we need to choose  a
class $z \in H^2(W,\Z)\cong\Z\oplus\Z$ with $z{|\partial W} = u $.
For this we let   $m,n$ be integers with $a\,m + b\,n = 1$ and set
$z = -n\,x + m\,y$.  It follows that $\gs^*(z)=u$ and thus $z$ has
the required properties in \eqref{KSzc}.

Note that in the case of $b$ even, we obtain $w_2(T\bar W)  \equiv
(1+a)\,x + b\,y  \mod 2$ as required. However,  in the case of $b$
odd, we have $w_2(T\bar W) \equiv c + z \mod 2 \equiv -n \,x + m\,y
\mod 2$.  For this to be equivalent to $ (1+a)\,x + b\,y \mod 2$, we
additionally need to choose $m$ to be odd (which one easily sees
implies $n \equiv 1 +a \mod 2$ as required). Note that it is always
possible to choose such an $m$.

Replacing the value of $p_1$ in \eqref{KSspin} and \eqref{KSnonspin}
we obtain

\begin{center}
\vspace{5pt} $b$ even
\end{center}
$$\begin{aligned}
S_1(\bar W,z) &= -\frac{1}{2^5 \cdot 7}\,\sign(W)+ \frac{1}{2^7\cdot 7}\,((3+4\,t)^2\,x^4 + 2\,(3+4\,t)\,x^2\,e^2 + e^4),
                     \\
S_2(\bar W,z) &=-\frac{1}{2^4\cdot 3}\,((3 + 4\,t)\,z^2\,x^2 + z^2\,e^2 - 2\,z^4),
                      \\
 S_3(\bar W,z)
                     &=-\frac{1}{2^2\cdot 3}\,((3 + 4\,t)\,z^2\,x^2 + z^2\,e^2 - 8\,z^4).\  \\
\end{aligned}$$
\begin{center}
\vspace{5pt} $b$ odd
\end{center}
$$\begin{aligned}
 S_1(\bar W,z) &= -\frac{1}{2^5 \cdot 7}\,\sign(W)   + \frac{1}{2^7\cdot 3 \cdot 7}\,[3\,(3+4\,t)^2\,x^4 + 6\,(3+4\,t)\,x^2\,e^2 + 3\,e^4  \\
 &    \qquad   \qquad    \qquad    \qquad    \qquad    \qquad    \quad        - 14\,(3+4\,t)\,z^2\,x^2 - 14\,z^2\,e^2 + 7\,z^4],
                     \\
S_2(\bar W,z) &=-\frac{1}{2^3\cdot 3}\,((3 + 4\,t)\,z^2\,x^2 + z^2\,e^2 - 5\,z^4),
                      \\
 S_3(\bar W,z)
                     &=-\frac{1}{2^3}\,((3 + 4\,t)\,z^2\,x^2 + z^2\,e^2 - 13\,z^4).\
  \end{aligned}$$

 We choose an orientation for $\bar N_t$ such that $\langle x^2\,y, [\bar N_t]\rangle = 1$.
The orientation for the vector bundle $E^2$ defines a Thom class
 $U \in H^2(W,\partial W)\cong\Z$ and we define the orientation on $\bar W$ such that
$U\cap [\bar W,
 \partial \bar W]= [\bar N_t]$.  On
  $\bar M^t_{a,\, b } = \partial \bar W$ we pick the orientation induced by the orientation
  on $\bar W$. Using $j^*(U) = e(E^2)=e$ with $j: \bar W \longrightarrow (\bar W,\partial \bar W)$, one easily  computes the characteristic numbers:
$$\begin{aligned}
e^4 &=  b\,(3\,a^2 + b^2\,t) \\
e^2\,z^2 &=  -2 \,a \,n \,m + b\, n^2 + b \,t\,m^2\\
x^2\,e^2 &=   b \\
\end{aligned}$$
For the characteristic numbers involving $x$ and $z$ note that $sx^2
=  (a\,x -b\,y) \cup e$ and $s\, z^2 = (\alpha \,x - \beta \,y) \cup
e$, where $ \alpha = a\,t\,m^2 + a\,n^2 + 2\,t\,b\,n\,m$ and $\beta
= b\,t\,m^2 + b\,n^2 + 2\,a\,n\,m$.
 Thus
$$\begin{aligned}
s\, z^4&=  -2\,n\,m\,\alpha - (n^2 + t\,m^2)\,\beta\\
 s\, z^2\,x^2 &=  - \beta\\
 s\, x^4 &=   -b
 \end{aligned}.$$

 One easily shows that the signature matrix is given by
 $
 \left(
   \begin{array}{cc}
     b & a \\
     a & bt \\
   \end{array}
 \right)
 $
 and since $\det R=-s$ and $\tr R=b(t+1)$, we obtain $\sign W$ as claimed. Substituting into \eqref{KSspin} and \eqref{KSnonspin} finishes the proof.
 \end{proof}

\smallskip

Recall that   the linking form is equal
to the characteristic number $z^4$ and we hence obtain:

\begin{cor}\label{linkingspin}
If $s =  a^2  - b^2\,t$ and $r=|s|>1$, the linking form of $\bar M^t_{a,\, b }$ is given by $$lk(\bar M^t_{a,\, b })=
-\frac{1}{s} \,[b\,n^4 + 6\,b\,t\,n^2\,m^2 + 4\,a\,t\,n\,m^3 + 4\,a\,n^3\,m + b\,t^2\,m^4]
\in \Q/\Z$$  where
 $m,n \in \Z$ are chosen such that $a\,m + b\,n = 1$, and if $b$ is odd, $m$ is odd as well.
\end{cor}

\begin{rem*}
The diffeomorphism classification of the manifolds $L_{a,b}=\bar
M^0_{a,b}$ was carried out in \cite{KS1}. For comparison, note that
\pref{samespin} implies that $x,y\in H^2(\bar N_t,\Z)$ are the
transgressions in the fiber bundle $\T^2\to \bar P_t\to\bar N_t$ of
the natural basis of $H^1(\T^2,\Z)$ corresponding to the splitting
$\T^2=\diag (e^{i\theta},e^{i\psi})\subset\S^1\times\S^3$. This is
the basis of $H^2(\bar N_t,\Z)$ used in \cite{KS1}.  Notice though
that in their notation $a$ and $b$ need to be switched (and hence
$n$ and $m$  as well).
\end{rem*}

\bigskip

\section{Comparison of invariants}\label{examples}

\bigskip

In this Section we discuss various diffeomorphisms that one obtains
by comparing Kreck-Stolz invariants. Our main interest are
diffeomorphisms with positively curved Eschenburg spaces, which can
only exist for the manifolds $S_{a,b}\,,\,M^t_{a,b}$ and some of the $\bar M^t_{a,b}$ ($b$ even).
 Finally we also discuss diffeomorphism between various manifolds that admit
Einstein metrics.

\subsection{Sphere bundles $\mathbf{ S_{a,b}}$}

In order to find sphere bundles $S_{a,b}$ diffeomorphic to a positively curved Eschenburg space we use the following strategy.
 As the Kreck-Stolz invariants for the Eschenburg spaces are quite complicated we compare sphere bundles $S_{a,b}$ to
  some fixed Eschenburg space $E_{k,l}$.    We specify the invariants of $E_{k,l}$ as follows:
$$
r=|\sigma_2(k)-\sigma_2(l)|  \ \   ;  \  s_i(E_{k,l}) \equiv \frac{A_i}{B_i} \mod 1 \ ; \  E_1:=\frac{224 r}{B_1}\,A_1 \ ; \  E_2:= \frac{24 r}{B_2}\,A_2 \ ; \  E_3:=  \frac{6 r}{B_3}\,A_3
$$
where  $A_i,\,B_i \in \Z,\,(A_i,B_i) = 1$ for $i=1, 2, 3$.

\bigskip

\no We now describe under which conditions they are diffeomorphic to sphere bundles.

\begin{thm}\label{Ediffeo}
A positively curved Eschenburg space $E_{k,l}$ is diffeomorphic to an $\Sph^3$ bundle over $\CP^2$ if and only if:
\begin{itemize}
  \item[(a)]  $ 224 \, r  \equiv 0 \mod B_1  \ ; \ 24 \,r \equiv  0 \mod B_2  \ \text{ and }\  6 \, r  \equiv 0 \mod B_3  $.
  \item[(b)] $E_1, E_2, E_3+1$ are even integers, and $E_3-E_2-3 \equiv  0 \mod 3 r$.
  \item[(c)]  $r+E_1 \text{ has a square root mod } 224\,  r$ such that
    $[r+E_1]^{1/2}_{224\, r} + E_2 -1 \equiv 0 \mod 8\, r$.
\end{itemize}
If these conditions are satisfied,  $E_{k,l}$  is diffeomorphic to $S_{a,a-r}$  if and only if
 $$a \equiv \frac{r+ 15\, [r+E_1]^{1/2}_{224\, r}}{2} +7 E_2 -8 \mod 168\, r  $$   for any square root satisfying (c).
If $r$ is negative, the diffeomorphism is orientation reversing.
\end{thm}
\begin{proof}
Using \pref{KSsphere} and \tref{KSdiffeo}, it follows that  $E_{k,l}$  is diffeomorphic to $S_{a,a-r}$ if and only if
\begin{equation}\label{equal}
\begin{aligned}
\frac{A_1}{B_1} &-   \frac{(2a-r+2)^2-r}{224r}=:x\in\Z  \\
\frac{A_2}{B_2} &+ \frac{(2a-r+1)}{24r}=:y\in\Z  \\
\frac{A_3}{B_3} &+ \frac{(2a-r-2)}{6r}=:z\in\Z
  \end{aligned}
  \end{equation}
  This implies in particular the divisibility condition in part (a). We  rewrite \eqref{equal} as follows:
  \begin{equation}\label{equal2}
\begin{aligned}
(2a-r+2)^2 &= E_1+r-224 r x   \\
 2a-r+1 &= -E_2+24ry  \\
 2a-r-2 &= -E_3+6rz
  \end{aligned}
  \end{equation}
  Since $r$ is odd for an Eschenburg space and since $a$ is an integer, this implies that $E_1, E_2, E_3+1$ are even and that $E_1+r$ has a square root mod $ 224r$.
  A solution to the second equation in \eqref{equal2} is a solution to the third equation if and only if $6r (4y-z)=3+E_2-E_3$. Thus, if the divisibility condition in part (b) is satisfied, we can find an integer $z$ for any integer solution $y$. We therefore set $$a= - \frac{1}{2} (E_2-r+1)+ 12 r y.$$ Next we observe that $e^2\equiv f^2 \mod 4n$ implies that $e\equiv \pm f \mod 2n$. Thus, if we let $S$ be a particular choice of a square root of $E_1+r$ mod $ 224r$,  \eqref{equal2}  implies that  $$1-E_2+24ry=S+112rx'$$ for some $x'\in\Z$ and hence $$S+E_2-1=8r( 3y-14x') .$$ This implies the divisibility condition in part (c) and if we let $ \alpha = \frac{1}{8\, r} \, ( S + E_2 -1 )$, we get $3y=\alpha+14x'$ or $y=5\alpha+14x''$. Thus
  $$
  a= - \frac{1}{2} (E_2-r+1)+ 60 r\alpha + 168rx''=- \frac{1}{2} (E_2-r+1)+ \frac{15}{2} \, ( S + E_2 -1 ) + 168rx''
  $$
  which finishes our proof.

\end{proof}

Since there are infinitely many Eschenburg spaces (not necessarily
positively curved) for a given $r$, it is conceivable that every
$\Sph^3$ bundle over $\CP^2$ is diffeomorphic to some Eschenburg
space. But in \cite{CEZ} it was shown that  for a given $r$, there
are only finitely many positively curved Eschenburg spaces. Our main
interest are diffeomorphisms of sphere bundles with positively
curved Eschenburg spaces.  We  use the computer program in
\cite{CEZ} to compute the invariants $\frac{A_i}{B_i}$ for an
Eschenburg space and write another program to find spaces satisfying
the conditions in \tref{Ediffeo}. A preliminary strong restriction
is used in a search, since the computation of the invariants
$\frac{A_i}{B_i}$ is time consuming.
 According to \cref{linkingS3}, a sphere bundle has standard linking form.
 This implies that only Eschenburg spaces with $\sigma_3(k)-\sigma_3(l)=\pm 1\mod r$
 can possibly be sphere bundles. This turns out to be a very strong restriction.

\smallskip

The invariants for Aloff-Wallach spaces have much simpler
expressions and there are many solutions where $W_{p,1}$ is
diffeomorphic to a sphere bundle, in addition to the natural
diffeomorphisms in \pref{sameAW}.  For example, for  $W_{1,1}$ one
has $s_1=\frac{1}{112}, s_2=\frac{-1}{36}, s_3=\frac{1}{18}$ and
hence $E_1=6, E_2=-2, E_3=1$. There are $ 8$ square roots of
$E_1+r=9$ mod $224\,r$, but only the values $3$ and $627$ satisfy
the divisibility condition. For $r<0$ there are no solutions. Thus
$W_{1,1}$ is diffeomorphic to a sphere bundle $S_{a,a-r}$ if and
only if $a\equiv 2 \text{ or } 146\mod 504$, whereas only the
diffeomorphism with $S_{2,-1}$ is a natural one.

\smallskip

There are also Aloff-Wallach spaces other than $W_{p,1}$ which are diffeomorphic to a sphere bundle. For example, $W_{56 , 103}$, and thus $r=19 513$,  is  diffeomorphic to $S_{a,a-r}$   if and only if
 $a  \equiv 273 181 \mod
3 278 184 $.

\smallskip

Among the $14 388$ positively curved Eschenburg spaces for $r<1000$  there are, in addition
  to the $19$ Aloff-Wallach spaces $W_{p,1}$,  $13$ which satisfy the conditions in \tref{Ediffeo}, see
  Table \ref{SBE}. We can now turn this around and state that among sphere bundles with $r<1000$, only the
  ones listed in Table \ref{SBE} are diffeomorphic to a positively curved Eschenburg
  space.

  In the Tables an entry  $r^*$ indicates that the
  diffeomorphism is orientation reversing.

\smallskip

\subsection{Circle bundles $\mathbf{M^t_{a,b}}$}
In this case one is not able to obtain a simple characterization of
which Eschenburg spaces are diffeomorphic to such circle bundles
since the invariants in \pref{KSCircle} are too complicated. As we saw
before such diffeomorphisms can only exists when $t$ is odd
since $\pi_4(M^{2t}_{a,b})=\Z_2$. We thus limit the search to sample
diffeomorphisms. Recall though that there are a number of natural
diffeomorphism of ${M^t_{a,b}}$ with the positively curved
Aloff-Wallach spaces. There are also many other diffeomorphisms with
Aloff-Wallach spaces. For example, $W_{1,1}$ is diffeomorphic to
$M^t_{a,b}$ with $[a,b,t]= [-11 ,2 , 35] , [-21 ,1 ,503 ] $ or $
[-25 ,1 ,647]$. Similarly, for the circle bundles $F_{p,q}$. For
example, $F_{3,1}$ is diffeomorphic to $M^t_{a,b}$ with $[a,b,t]=
[-189 ,4 ,2281 ]  $ or $[-111 ,4 ,799 ]$.

To find other diffeomorphisms we fix bounds $A,B$,  produce a list
of all circle bundles $M^t_{a,b}$ with $r<A$ and  $|a|,|b|<B$
(letting $t$ become large) and compute their Kreck-Stolz invariants.
Similarly, for all positively curved Eschenburg spaces with  $r<A$.
We then compare the two lists to produce diffeomorphic pairs. In
case where the bounds are $r\le 101,\ |a|,|b|\le 1000$, there are a
total of 316 diffeomorphism which are not of natural type, 301 are
with Aloff-Wallach spaces, 10 with the bundles $F_{p,q}$ and the
remaining 5 are listed in Table \ref{CBE}.

Although $\bar M^t_{a,b}$ is also a spin manifold, a Maple search
did not produce any examples which are diffeomorphic to positively
curved Eschenburg spaces.

\smallskip

\subsection{Einstein manifolds}
Among the manifolds we discussed, there are 3 classes which are
known to admit Einstein metrics:

\begin{itemize}
\item \cite{W}\ \ \ The homogeneous Aloff-Wallach
spaces $M^1_{a,b}=W_{a,b}$.
\medskip
\item \cite{WZ} The circle bundles $\bar M^0_{a,b}=L_{a,b}$ with base
$\CP^2\times\CP^1$.
\medskip
\item \cite{Che} The
3-sphere bundles $S_{a,b}$, $\bar S_{a,b}$ whose structure group
reduces to $\T^2\subset\SO(4)$.
 \end{itemize}
 From the discussion in Section 1, it
follows that  a reduction to $\T^2$ is possible if and only if
$p_1(P_-)$ and $p_1(P_+)$ are squares. To relate our notation to the
one in \cite{Che}, recall that under the two fold cover a circle of
slope $(a,b)$ in $\S^3\times\S^3$ is sent to a circle of slope
$(a-b,a+b)$ in $\SO(4)$. This easily implies that the manifolds in
\cite{Che}, parametrized by $q_1,q_2\in\Z$ in his notation, are the
manifolds $S_{a,b}$ with $4a+1=(q_1+q_2)^2$ and $4b+1=(q_1-q_2)^2$
in the case of $q_1+q_2$ odd, and the manifolds $\bar S_{a,b}$ with
$4a=(q_1+q_2)^2$ and $4b=(q_1-q_2)^2$ in the case of $q_1+q_2$ even.
In both cases $r=a-b=q_1q_2$. For convenience, we denote these
Einstein manifolds by $C_{q_1,q_2}$ if $q_1+q_2$ is odd, and $\bar
C_{q_1,q_2}$ if $q_1+q_2$ is even.

\smallskip
Comparisons of Einstein manifolds within each class have been
carried out before. In the second and third case the diffeomorphism
classification is given by simple congruences, see \cite{KS1,Che},
and each space is diffeomorphic to infinitely many other Einstein
manifolds.   For example:

\begin{itemize}
\item The Einstein manifold $L_{a,b}$ is diffeomorphic to $L_{a,b'}$ if  $b\equiv b' \mod 56\,
a^2$.
\item  The Einstein manifold $C_{q_1,q_2}$ (resp. $\bar C_{q_1,q_2}$) is diffeomorphic
to $C_{q_1',q_2'}$ (resp. $\bar C_{q_1',q_2'}$) if
$r=q_1q_2=q_1'q_2'$  and $q_1^2+q_2^2\equiv q_1'^2+q_2'^2\ \mod
672 \, r$.
 \end{itemize}

\smallskip
For the Aloff-Wallach spaces on the other hand diffeomorphisms among
each other are very rare, see \cite{KS2}.

Using our results, we can compare Einstein manifolds that belong to
different classes. We first observe:

\smallskip

\begin{itemize}
\item There are no diffeomorphisms between the spin  Einstein manifolds
$C_{q_1,q_2}$ and either $W_{a,b}$ or $L_{a,2b}$ since in the
first case $r$ is even, and in the other two cases $r$ is odd.
\medskip
\item There are no diffeomorphisms between the spin Einstein manifolds
 $W_{a,b}$ and $L_{a,2b}$ since in the first case $p_1=0$ and in
 the second case $p_1=3(2b)^2\mod a^2$ which can never be $0$
 since $(a,b)=1$ (see \pref{circletop} and
 \pref{circletopspin}).
 \end{itemize}

\medskip

But among the non-spin Einstein manifolds $\bar C_{q_1,q_2}$ and
$L_{a,2b+1}$ there are some diffeomorphisms:

\smallskip

\begin{itemize}
\item The Einstein manifold  $\bar C_{q,q}\simeq \bar S_{q^2,\,0}$
 is naturally diffeomorphic to the Einstein manifold $L_{q,\,
1}$ (see \pref{diffspin}).
\medskip
\item The Einstein manifold $\bar C_{10,\, 490}\simeq \bar S_{62500,\, 57600}$, is
 diffeomorphic to the Einstein manifold $L_{ 70,\, 5899}$.
 \end{itemize}

 \medskip
Further examples are difficult to find. Indeed, if a diffeomorphism
between $C_{q_1,q_2}$ and $L_{a,b}$ exists, then $q_1q_2$ must be a
square, and the linking form of $L_{a,b}$ must be standard. Both of
these are strong restrictions.

 \medskip

 We finally remark that the spin Einstein manifolds $C_{q_1,q_2}$
 can never be diffeomorphic to an Eschenburg space since in the
 first case $q_1+q_2$ is odd, and hence $r=q_1q_2$ is even, whereas
  for Eschenburg space $r$ is always odd. Notice also that $L_{a,\,
  b} $ can never be diffeomorphic to an Eschenburg space since
  $\pi_4(L_{a,\, b})=\Z_2$.

  \smallskip

 For the convenience of the reader,  the Maple program
 that computes the invariants is available at
 www.math.upenn.edu/wziller/research.

\renewcommand{\thetable}{\Alph{table}}

\renewcommand{\arraystretch}{1.6}
 {\setlength{\tabcolsep}{0.2cm}
\stepcounter{equation}
\begin{table}[htpb]
\begin{center}
\begin{tabular}{|c|c|c|c|c|c|}
\hline
$r$ & $[k_1,k_2,k_3 \;\vert\; l_1,l_2,l_3]$& $[a ] \mod 168\, r$ & $s_1$ &
$s_2$ & $s_3$ \\
\hline \hline   $41^{*}$ & $[2,3,7\;\vert\;  12,0, 0]$ & $[2285, 5237]$ & 115/287 & 65/164 & -33/82 \\
\hline   $127$  &  $[17,16,-7\;\vert\;  14, 12, 0]$ & $[17230]$ &  3489/14224 & -403/1524 & -41/762 \\
\hline  $233$ & $[5,3,-31 \;\vert\;  -23,0, 0]$ &  $[2943, 36495]$ & -1863/6524 & -31/2796 &  -59/1398 \\
\hline $289^*$ & $[21, 18, -13 \;\vert\; 16, 10, 0]$ &  $[21194, 42002 ]$ & -397/1156 & 121/1734 & 481/1734 \\
\hline $611$ & $[25, 17, -23\;\vert\;  14,5, 0]$ &  $[69423, 84087 ]$ &-15789/68432& -1565/3666& 1075/3666\\
\hline $617$ & $[24, 19, -23\;\vert\;  14,6, 0]$ & $[13030  ]$ &-3567/8638& 1043/3702& 473/3702 \\
\hline $661$ & $[23, 21, -26\;\vert\;  18,0, 0]$ & $[56346, 72210  ]$ &1787/37016& -41/661& -327/1322 \\
\hline $673^*$ & $[25, 14, -25\;\vert\;  8, 6, 0]$ & $[49154,81458  ]$ &-529/2692& 181/4038& 721/4038 \\
\hline $751^* $ & $[33, 33, -20\;\vert\; 26, 20, 0]$ & $[7036, 43084  ]$ &-12629/84112& -2351/9012& -199/4506 \\
\hline $911^*$ & $[69, 65, -13\;\vert\; 63, 58, 0]$ & $[123457, 145321 ]$ & -7445/102032& 1375/5466& 31/5466 \\
\hline$ 911^*$ & $[23, 23, -31\;\vert\; 14, 1, 0]$ & $[1457, 132641 ]$ & 31083/102032& 167/1822& 667/1822 \\
\hline $929^*$ & $[41, 17, 4\;\vert\;  62, 0, 0]$ & $[22359, 44655 ]$ &1441/6503& 401/11148& 805/5574\\
\hline $991^*$ & $[51, 45, -19\;\vert\; 43, 34, 0]$& $[18113,89465] $ & -44333/110992& 2863/5946& -443/5946 \\
\hline
\end{tabular}
\end{center}
 \vspace{0.2cm}
     \caption{  Sphere bundles $S_{a,a-r}$  with $r<1000$ which are diffeomorphic to positively curved Eschenburg Spaces
     $E_{k,l}$ other than $W_{p,1}$.}\label{SBE}
 \end{table}

\renewcommand{\arraystretch}{1.4}
 {\setlength{\tabcolsep}{0.2cm}
\stepcounter{equation}
\begin{table}[htpb]
\begin{center}
\begin{tabular}{|c|c|c|c|c|c|}
\hline
$r$ & $[k_1,k_2,k_3 \;\vert\; l_1,l_2,l_3]$& $[a,b,t]$ & $s_1$ &
$s_2$ & $s_3$ \\
\hline
\hline  17 & $[1, 2, 5\;\vert\;  8,0, 0]$ & $[638 ,-607 ,-403 ]$ &-201/952& 55/204& 23/102 \\
\hline  25 & $[1,2,-9\;\vert\;  -6,0, 0]$ & $[621, -614, -7781 ]$ &19/50& -3/10& -17/50\\
\hline  33 & $[1, 1, 16\;\vert\;  18,0, 0]$ & $[805, -632, -17 ]$ & 47/308& -125/396& 53/198 \\
\hline  $41^*$ & $[2,3,7 \;\vert\;  12,0, 0]$ & $[580, -579, -335861 ]$&  -115/287& -65/164& 33/82\\
\hline  $41^*$ & $[2,3,7\;\vert\;  12,0, 0]$ & $[405, -404, -163661 ]$ &-115/287& -65/164& 33/82 \\
\hline
\end{tabular}
\end{center}
 \vspace{0.2cm}
     \caption{ Circle bundles $M^t_{a,\, b }$, other than $W_{p,q},\ F_{p,q}$, diffeomorphic to positively curved Eschenburg Spaces
     $E_{k,l}$.}\label{CBE}
 \end{table}

\vspace{10cm}

\providecommand{\bysame}{\leavevmode\hbox
to3em{\hrulefill}\thinspace}

\end{document}